\documentclass[12pt]{amsart}
\usepackage{tikz-cd}
\usepackage{enumitem}
\setlist[description]{leftmargin=\parindent,labelindent=\parindent}
\usepackage{amssymb}
\usepackage{amsmath,amscd}
\usepackage{mathrsfs}
\usepackage{url}
\usepackage{mathtools}
\usepackage{cite}
\usepackage{fullpage}
\usepackage{hyperref}
\usepackage{verbatim}
\usepackage{comment}
\usepackage{fancyvrb}
\usepackage{fvextra}
\usepackage{caption}
\usepackage{tikz}
\usepackage{tkz-graph}
\usetikzlibrary{matrix}
 \usetikzlibrary{shapes}
\usetikzlibrary{arrows,decorations.markings}
\usepackage{tikz-cd}
\usepackage{stmaryrd}

\pgfarrowsdeclare{bad to}{bad to}
{
  \pgfarrowsleftextend{-2\pgflinewidth}
  \pgfarrowsrightextend{\pgflinewidth}
}
{
  \pgfsetlinewidth{0.8\pgflinewidth}
  \pgfsetdash{}{0pt}
  \pgfsetroundcap
  \pgfsetroundjoin
  \pgfpathmoveto{\pgfpoint{-3\pgflinewidth}{4\pgflinewidth}}
  \pgfpathcurveto
  {\pgfpoint{-2.75\pgflinewidth}{2.5\pgflinewidth}}
  {\pgfpoint{0pt}{0.25\pgflinewidth}}
  {\pgfpoint{0.75\pgflinewidth}{0pt}}
  \pgfpathcurveto
  {\pgfpoint{0pt}{-0.25\pgflinewidth}}
  {\pgfpoint{-2.75\pgflinewidth}{-2.5\pgflinewidth}}
  {\pgfpoint{-3\pgflinewidth}{-4\pgflinewidth}}
  \pgfusepathqstroke
}

\newcommand{\evec}{{\vec{e}} \nobreak\hspace{.16667em plus .08333em}'}

\newtheorem{thm}{Theorem}[section]
\newtheorem{prop}[thm]{Proposition}
\newtheorem{lem}[thm]{Lemma}
\newtheorem{cor}[thm]{Corollary}
\newtheorem{conj}[thm]{Conjecture}

\theoremstyle{definition}

\newtheorem{example}[thm]{Example}
\newtheorem{rem}[thm]{Remark}
\newtheorem{question}[thm]{Question}

\numberwithin{equation}{section}

\newcommand{\U}{\mathcal{U}}
\newcommand{\cO}{\mathcal{O}}
\newcommand{\X}{\mathcal{X}}
\newcommand{\Y}{\mathcal{Y}}

\newcommand{\J}{\mathcal{J}}

\newcommand{\Z}{\mathcal{Z}}

\newcommand{\QQ}{\mathbb{Q}}
\newcommand{\CC}{\mathbb{C}}
\newcommand{\RR}{\mathbb{R}}

\newcommand{\pp}{\mathbb{P}}
\newcommand{\PP}{\mathbb{P}}

\renewcommand{\O}{\mathcal{O}}

\DeclareMathOperator{\Aut}{Aut}

\DeclareMathOperator{\Pic}{Pic}
\DeclareMathOperator{\Spec}{Spec}

\DeclareMathOperator{\coker}{coker}

\DeclareMathOperator{\Sym}{Sym}

\DeclareMathOperator{\Hom}{Hom}
\DeclareMathOperator{\Stab}{Stab}
\DeclareMathOperator{\End}{End}

\DeclareMathOperator{\diag}{diag}

\DeclareMathOperator{\SUT}{\mathsf{SU}}
\renewcommand{\d}{\mathrm{d}}
\renewcommand{\gg}{\mathbb{G}}
\renewcommand{\evec}{{\vec{e}} \nobreak\hspace{.16667em plus .08333em}'}

\renewcommand{\k}{\mathbf{k}}
\usepackage{wrapfig}

\renewcommand{\S}{\mathcal{S}}

\renewcommand{\nu}{v}

\usepackage{mathtools}
\usepackage{mathalpha}
\newcommand*{\Lcorner}{%
    \mathchoice%
        {\mathrel{\makebox[7pt][c]{\rule{.7pt}{7.5pt}\rule{5pt}{.7pt}}}}%
        {\mathrel{\makebox[7pt][c]{\rule{.7pt}{7.5pt}\rule{5pt}{.7pt}}}}%
        {\mathrel{\makebox[5.5pt][c]{\rule{.7pt}{5.25pt}\rule{3.5pt}{.7pt}}}}%
        {\mathrel{\makebox[4pt][c]{\rule{.7pt}{3.75pt}\rule{2.5pt}{.7pt}}}}%
}

\renewcommand{\llcorner}{\Lcorner}

\title{Brill--Noether theory of smooth curves in the plane and on Hirzebruch surfaces}
\author{Hannah Larson}
\author{Sameera Vemulapalli}

\begin{document}

\maketitle

\begin{abstract}
In this paper, we describe the Brill--Noether theory of a general smooth plane curve and a general curve $C$ on a Hirzebruch surface of fixed class. It is natural to study the line bundles on such curves according to the splitting type of their pushforward along 
projection maps $C \to \pp^1$. Inspired by Wood's parameterization of ideal classes in rings associated to binary forms, we further refine the stratification of line bundles $L$ on $C$ by fixing the splitting types of both $L$ and $L(\Delta)$, where $\Delta \subset C$ is the intersection of $C$ with the directrix of the Hirzebruch surface. Our main theorem determines the dimensions of these locally closed strata and, if the characteristic of the ground field is zero, proves that they are smooth.
\end{abstract}

\section{Introduction}

Brill--Noether theory studies maps from curves to projective spaces.
Given a curve $C$, the data of a degree $d$ map $C \to \pp^r$ is equivalent to the data of a degree $d$ line bundle on $C$ together with a choice of $r+1$ global sections having no common zeros. As such, a central object of study is the
\emph{Brill--Noether locus}
\[W^r_d(C) := \{L \in \Pic^d(C): h^0(L) \geq r+1\},\]
parameterizing the line bundles giving rise to such maps.
When $C$ is a \emph{general} curve of genus $g$, the famous Brill--Noether theorem gives a nice description of these varieties: $W^r_d(C)$ has the ``expected dimension"
$\rho(g, r, d) := g - (r + 1)(g - d + r)$ \cite[Main Thereom II]{GH}
and is irreducible when its dimension is positive \cite[Corollary 1]{FL}. Moreover, the open subset $U^r_d(C) := W^r_d(C) \smallsetminus W^{r+1}_d(C)$ of line bundles with exactly $r+1$ global sections is smooth and is dense within $W^r_d(C)$ \cite[Theorem 1.1]{GP}.

However, the curves that we come across in nature often already live in (or map to) some projective space, and the very presence of such a map may mean our curve fails the Brill--Noether theorem! The first case of this is curves already equipped with a low-degree map to $\pp^1$. If $\alpha: C \to \pp^1$ is a general degree $k$ cover, then $W^r_d(C)$ could have multiple components of varying dimensions \cite[Corollary 1.3]{refined}.

The key to understanding these components is to study a finer invariant of line bundles called the \emph{splitting type}.
If $L$ is a line bundle on $C$, then $\alpha_*L$ is a rank $k$ vector bundle on $\pp^1$. By the Birkoff--Grothendieck theorem, every vector bundle on $\pp^1$ is a direct sum of line bundles.
The splitting type of $L$ is the tuple of integers $\vec{e} = (e_1, \ldots, e_k)$ with $e_1 \leq \cdots \leq e_k$ such that $\alpha_*L \cong \O(e_1) \oplus \cdots \oplus \O(e_k)$. We often abbreviate the corresponding sum of line bundles on $\pp^1$ by $\O(\vec{e})$.
The ``expected codimension" of a splitting type is
\[u(\vec{e}) := h^1(\pp^1, \mathrm{End}(\O(\vec{e})) = \sum_{i, j} \max\{0, e_i - e_j - 1\}.\]
Given two splitting types $\evec$ and $\vec{e}$, we write $\evec \leq \vec{e}$ if $\evec$ has the potential to arise as a specialization of $\vec{e}$, i.e. if $e_1' + \ldots + e_j' \leq e_1 + \ldots + e_j$ for all $j$ and equality holds for $j = k$.
We then define the \emph{Brill--Noether splitting locus}
\[W^{\vec{e}}(C) := \{L \in \Pic^d(C) : \alpha_*L \cong \O(\evec) \text{ for some $\evec \leq \vec{e}$}\}.\]
When $C \to \pp^1$ is a \emph{general} degree $k$, genus $g$ cover, the Hurwitz--Brill--Noether theorem gives a nice description of these Brill--Noether splitting loci: $W^{\vec{e}}(C)$ has the ``expected dimension" $\rho'(g, \vec{e}) := g - u(\vec{e})$ \cite[Theorem 1.2]{refined} and is irreducible when its dimension is positive \cite[Theorem 1.2]{llv}. Moreover, the open subset 
\[U^{\vec{e}}(C) := \{L \in \Pic^d(C) : \alpha_*L \cong \O(\vec{e})\}\]
of line bundles with splitting type $\vec{e}$ is smooth \cite[Theorem 1.2]{refined} and it is dense within $W^{\vec{e}}(C)$ \cite[Lemma 2.1]{refined}.

In this paper, we consider the natural next case of curves equipped with a map to the plane, provided that the image is smooth. If $\iota: C \hookrightarrow \pp^2$ is a smooth, degree $k$ plane curve, then projection from any point not on $C$ defines a degree $k$ map $\alpha: C \to \pp^1$. These projections are evidently special degree $k$ covers, and so the Hurwitz--Brill--Noether theorem does not apply. Nevertheless, the splitting type of line bundles with respect to these projections is a fruitful invariant to study. Note that the splitting type of $L$ with respect to $\alpha: C \to \pp^1$ is equivalent to the data of $h^0(C, L \otimes \alpha^*\O_{\pp^1}(1)) = h^0(C, L \otimes \iota^*\O_{\pp^2}(1))$, and thus is is independent of the choice of projection $\alpha$. Throughout, we work over an algebraically closed field.

\begin{thm} \label{planecurves}
Let $C \subset \pp^2$ be a smooth plane curve of degree $k$ and let $\alpha: C \to \pp^1$ be the projection from any point not on $C$. If $U^{\vec{e}}(C)$ is non-empty, then $e_{i+1} - e_i \leq 1$ for all $i$. If $e_{i+1} - e_i \leq 1$ for all $i$ and $C$ is general, then $U^{\vec{e}}(C)$ is non-empty of pure dimension
\[\dim U^{\vec{e}}(C) = g - \#\{(i, j) : e_i - e_j \geq 2\}.\]
If the characteristic of the ground field is zero, then $U^{\vec{e}}(C)$ is smooth for general $C$.
\end{thm}

Note that the splitting type of the structure sheaf is $(-k + 1, -k +2, \ldots, -1, 0)$. Therefore, the dimension zero splitting loci are precisely the structure sheaf and multiples of the $g^2_k$ (whose splitting types are shifts of that of the structure sheaf).
In contrast with the situation for general covers, for plane curves, $U^{\vec{e}}(C)$ need not be dense in $W^{\vec{e}}(C)$, as the following example illustrates!

\begin{example}
Suppose $C \subset \pp^2$ is a general degree $7$ plane curve, so its genus is $15$.
The splitting types that occur in $U^2_{14}(C)$ (degree $14$ line bundles with exactly $3$ sections) and the dimensions of these loci are listed below.
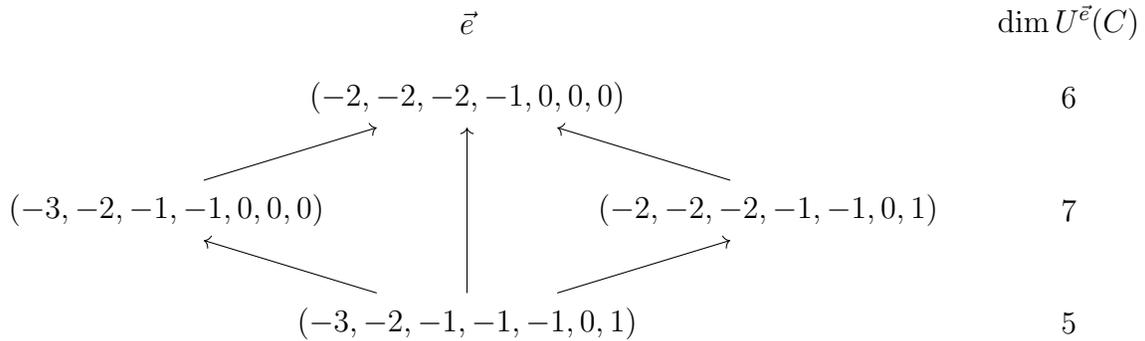
\begin{center}
\begin{tikzpicture}
\node at (0, 0) {$(-2, -2, -2, -1, 0, 0, 0)$}; 
\node at (-4, -1.5) {$(-3, -2, -1, -1, 0, 0, 0)$};
\node at (4, -1.5) {$(-2, -2, -2, -1, -1, 0, 1)$};
\node at (0, -3) {$(-3, -2, -1, -1, -1, 0, 1)$};
\draw[<-] (-1.2, -.4) -- (-3.5, -1.1);
\draw[<-] (1.2, -.4) -- (3.5, -1.1);
\draw[->] (-1.2, -3 +.4) -- (-3.5, -3+1.1);
\draw[->] (1.2, -3 +.4) -- (3.5, -3+1.1);
\node at (8, 0) {$6$};
\node at (8, -1.5) {$7$};
\node at (8, -3) {$5$};
\node at (0, 1) {$\vec{e}$};
\node at (8, 1) {$\dim U^{\vec{e}}(C)$};
\draw[<-] (0, -.4) -- (0, -3 + .4);
\end{tikzpicture}
\end{center}
In the diagram, there is an arrow drawn between splitting types $\evec$ and $\vec{e}$ when $\evec < \vec{e}$.
If $U^{\evec}(C)$ is contained in the closure of $U^{\vec{e}}(C)$ then semicontinuity theorems imply that $\evec \leq \vec{e}$. However, as evidenced by the dimensions above, the converse need not hold! In particular, we see that $W^{(-2, -2, -2, -1, 0, 0, 0)}(C)$ has at least three components, at least two of which are not contained in the closure of $U^{(-2, -2, -2, -1, 0, 0, 0)}(C)$. We are unsure how many of the strict splitting loci below contain the last splitting type in their closure.

\end{example}

\begin{rem}
The above example shows that $W^{\vec{e}}(C)$ may have many 
components of varying dimensions when $C$ is a smooth plane curve. Nevertheless, the strict splitting loci $U^{\vec{e}}(C)$ are pure-dimensional.
As we shall explain, Wood's parameterization implies that as we vary over all smooth plane curves, the universal $\U^{\vec{e}}$ is irreducible. Thus, even though the general fiber $U^{\vec{e}}(C)$ may be reducible, all components must be exchanged under monodromy.
\end{rem}

\subsection{Curves on Hirzebruch surfaces}
Theorem \ref{planecurves} is a special case of a more general result about smooth curves on Hirzebruch surfaces.
Indeed, the projection of a plane curve from a point not on it naturally factors through the blowup of $\pp^2$ at the point. Thus, we obtain our curve lying on the Hirzebruch surface $\mathbb{F}_1$ in such a way that it does not meet the exceptional divisor.

More generally, let $\mathbb{F}_m$ be the Hirzebruch surface $\pp(\O_{\pp^1}\oplus \O_{\pp^1}(m))$ and suppose $C \subset \mathbb{F}_m$ is a smooth curve. Composing with the projection $\mathbb{F}_m \to \pp^1$ gives rise to a degree $k$ map $\alpha: C \to \pp^1$.
Next, let $D \subset \mathbb{F}_m$ be the unique curve of self-intersection $-m$, called the \emph{directrix}. (If $m = 0$, then $\mathbb{F}_m \cong \pp^1 \times \pp^1$ and we define $D$ to be the class of the fiber of the other projection.)
The intersection $\Delta := C \cap D$ defines a distinguished effective divisor
on $C$ (or divisor class if $m = 0$). 
Note that the number $\delta := \deg(\Delta)$ determines the class of $C \subset \mathbb{F}_m$, as we have already decided $C$ meets a fiber of projection to $\pp^1$ in $k$ points. In turn, these quantities determine the genus by adjunction, namely (see \eqref{gf})
\[g = {k \choose 2} m + (k - 1)(\delta - 1).\] 

Given a line bundle $L$ on $C \subset \mathbb{F}_m$, 
we study the \emph{simultaneous} splitting types of $L$ and $L(\Delta)$. We define
\[U^{\vec{e}, \vec{f}}(C) := \{L \in \Pic(C) : \alpha_*L \cong \O(\vec{e}) \text{ and } \alpha_*L(\Delta) \cong \O(\vec{f})\}. \]
Clearly, we have 
\[U^{\vec{e}}(C) = \bigcup_{\vec{f}} U^{\vec{e},\vec{f}}(C),\]
where above the union runs over all $\vec{f}$.
The following example illustrates how, when $\Delta$ is non-empty, the $U^{\vec{e},\vec{f}}(C)$ refine the usual stratification by splitting type.
\begin{example} \label{trige}
Suppose $m = 3$, $k = 3$, and $\delta = 2$. We have $g = 11$. 
Let $\Delta = \{x, y\} \subset C$. Note that $C$ is a special trigonal curve since $m > 1$.
In \cite[Example 3.7]{trig}, it was shown that $U^{(-8, -4, -1)}(C)$ has three components, each isomorphic to the curve $C$ or $C$ minus a point:
\begin{align*}
C_1 &= \{x - p: p \in C \smallsetminus x\} \\
C_2 &= \{y - p: p \in C \smallsetminus y\} \\
C_3 &= \{x+ y + p - F : p \in C\}.
\end{align*}
Here, $F$ denotes the class of the fiber on $\mathbb{F}_m$, which is the $g^1_3$ on $C$. These components are pictured on the left below. The ``missing point" at the intersection of $\overline{C}_1$ with $\overline{C}_2$ corresponds to $[\O_C] \in U^{(-8, -5, 0)}(C)$.

Theorem \ref{maint} below tells us that $U^{(-8,-4,-1), \vec{f}}(C)$ is non-empty for the following $\vec{f}$
and determines their dimensions
\begin{center}
\begin{tabular}{c|c|c|c}
  $\vec{f}$ &
  ${\color{orange} (-7, -4, 0)}$
 & ${\color{violet} (-7, -3, -1)}$ &  ${\color{green!70!black} (-6,-4,-1)}$    \\[6pt]
  \hline 
  & & & \\[-8pt]
$\dim$   &  $0$ & $1$ & $1$ \\
  \hline 
  & & & \\[-8pt]
 $U^{(-8,-4,-1),\vec{f}}$ & $\{x-y,y-x\}$ & $C_1 \cup C_2 \smallsetminus \{x - y, y - x\}$ & $C_3 \smallsetminus (C_1 \cup C_2)$
\end{tabular}
\end{center}
The picture on the right labels the non-empty $U^{(-8,-4,-1), \vec{f}}(C) \subset U^{(-8,-4,-1)}(C)$ each by a different color according to their $\vec{f}$. 

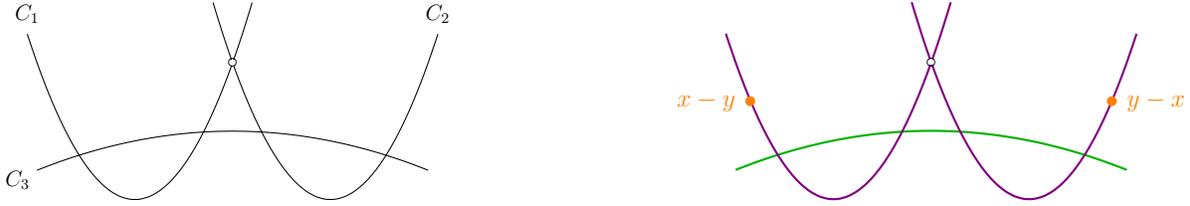
\begin{figure}[h!]
\centering
\begin{tikzpicture}[scale=1.3]
  \draw[domain=.8:3.1,smooth,variable=\x] plot ({\x},{1.4*(\x-2)*(\x-2)});
    \draw[domain=-1.1:1.2,smooth,variable=\x] plot ({\x},{1.4*(\x)*(\x)});
  \draw[domain = -1:3,smooth,variable=\x] plot ({\x},{-.1*(\x-1)*(\x-1) + .7});
  \node[scale=.5, color = white] at (1, 1.4) {$\bullet$};
    \node[scale=.7] at (1, 1.4) {$\circ$};
   \node[scale=.7] at (-1.1, 1.9) {$C_1$};
      \node[scale=.7] at (3.1, 1.9) {$C_2$};
       \node[scale=.7] at (-1.2, .2) {$C_3$};
\end{tikzpicture}
\hspace{1in}
\begin{tikzpicture}[scale=1.3]
    \draw[thick, domain = -1:3,smooth,variable=\x,black!30!green] plot ({\x},{-.1*(\x-1)*(\x-1) + .7});
  
  \draw[thick, domain=.8:3.1,smooth,variable=\x,violet] plot ({\x},{1.4*(\x-2)*(\x-2)});
    \draw[thick, domain=-1.1:1.2,smooth,variable=\x,violet] plot ({\x},{1.4*(\x)*(\x)});
  \node[scale=.5, color = white] at (1, 1.4) {$\bullet$};
    \node[scale=.7] at (1, 1.4) {$\circ$};

  \node[scale=.8, color = orange] at (-.85, 1) {$\bullet$};
    \node[scale=.8, color = orange] at (-1.3, 1) {$x - y$};
        \node[scale=.8, color = orange] at (3.3, 1) {$y - x$};

  \node[scale=.8, color = orange] at (2.85, 1) {$\bullet$};
  
\end{tikzpicture}
\label{f5}
\end{figure}
\noindent
Given a point $p \in C$, we write $p'$ and $p''$ for the two points such that $p + p' + p'' = F$.
Then, $C_1 \cap C_3 = \{x - y', x - y''\}$ and $C_2 \cap C_3 = \{y - x', y - x''\}$. These points are all colored purple on the right.
Notice that although each of $U^{(-8,-4,-1), {\color{violet} (-7,-3,-1)}}(C)$ and $U^{(-8,-4,-1), {\color{orange} (-7,-4,0)}}(C)$ has two components, their two components are exchanged by swapping $x$ and $y$. Thus, as we vary over the full linear system of curves of class $3H + 2F$ on $\mathbb{F}_m$, the components cannot be distinguished.
Finally, note that the three possibilities for $\vec{f}$ satisfy
\[{\color{orange} (-7, -4, 0)} \leq {\color{violet} (-7, -3, -1)} \leq {\color{green!70!black} (-6, -4, -1)},\]
so we ``expect" that each should lie in the closure of the next. 
However, the latter two have the same dimension, so the ${\color{violet} \vec{f} = (-7, -3, -1)}$ locus clearly cannot lie in the closure of the ${\color{green!70!black} \vec{f} = (-6, -4, -1)}$ locus.
More curious is that, even though the ${\color{orange} \vec{f} = (-7, -4, 0)}$ locus has smaller dimension than the ${\color{green!70!black} \vec{f} = (-6, -4, -1)}$ locus, the former does not lie in the closure of the latter.
\end{example}

Our main theorem determines the precise conditions under which $U^{\vec{e}, \vec{f}}(C)$ is non-empty for a general curve $C \subset \mathbb{F}_m$, and its dimension in that case. We also prove smoothness in characteristic zero.

\begin{thm} \label{maint}
Suppose $C \subset \mathbb{F}_m$ is a smooth curve of genus $g$ meeting the directrix in degree $\delta$. If $U^{\vec{e}, \vec{f}}(C)$ is non-empty, then
\begin{enumerate}
    \item $f_i \geq e_i$
    \item $f_i \geq e_{i+1} -m$
    \item $\sum (f_i - e_i) = \delta$.
\end{enumerate} 
Conversely, if the inequalities above are satisfied and $C$ is general, then $U^{\vec{e}, \vec{f}}(C)$ is non-empty of pure dimension
\[\dim U^{\vec{e}, \vec{f}}(C) = g - u(\vec{e}) - u(\vec{f}) + h^1(\pp^1, \mathrm{Hom}(\O(\vec{e}), \O(\vec{f})) \otimes (\O \oplus \O(m)).\]
If the characteristic of the ground field is zero, then $U^{\vec{e}, \vec{f}}(C)$ is smooth for general $C$.
\end{thm}

\begin{rem} \label{nec}
It is not hard to see the \emph{necessity} of the conditions $(1)$ and $(3)$ above. 
\begin{enumerate}
\item[(1)] has simple geometric explanation: Since $\Delta$ is effective, 
\[h^0(\pp^1, \O(\vec{f})(n)) = h^0(C, L(\Delta) \otimes \alpha^* \O_{\pp^1}(n)) \geq h^0(C, L \otimes \alpha^*\O_{\pp^1}(n)) = h^0(\pp^1, \O(\vec{e})(n))\]
for all $n$. This implies condition (1) on $\vec{e}$ and $\vec{f}$. 
\item[(3)] must hold for degree reasons. Using Riemann--Roch, one finds that the degree of $L$ is $e_1 + \ldots + e_k + k + g - 1$ and similarly, the degree of $L(\Delta)$ is $f_1 + \ldots + f_k + k + g - 1$. In particular, the difference in these degrees must be the degree of $\Delta$.
\end{enumerate}
For further discussion of Conditions (1) and (2) and speculation on generalizations of these conditions to other families of covers, see Section \ref{discuss}. 
\end{rem}

To unpack the dimension formula in Theorem \ref{maint}, we interpret the last term 
\[\nu(\vec{e},\vec{f}) := h^1(\pp^1, \mathrm{Hom}(\O(\vec{e}), \O(\vec{f})) \otimes (\O \oplus \O(m))\]
as a ``correction factor." When $\nu(\vec{e},\vec{f}) = 0$, the codimension of $U^{\vec{e}, \vec{f}}(C)$ is equal to the sum of the expected codimension for $L$ to have splitting type $\vec{e}$ and the expected codimension for $L(\Delta)$ to have splitting type $\vec{f}$.
Thus, $\nu(\vec{e}, \vec{f})$ measures the failure of the translate of $U^{\vec{e}}(C)$ under tensor product with $\O_C(\Delta)$ to meet $U^{\vec{f}}(C)$ transversally, plus any excess dimension the splitting loci had to begin with.

\begin{rem}
In the special case $\delta = 0$, we have $\vec{e} = \vec{f}$ and the curves in Theorem \ref{maint} are covers of $\mathbb{P}^1$ that live in the total space of the line bundle $\O(m)$. Such curves are also known as \emph{spectral curves}. Using the Hitchin correspondence for spectral curves, recent work of Nollau \cite[Theorem 1.4]{n} observes that, either $U^{\vec{e}}(C)$ is empty or has the dimension claimed in Theorem \ref{maint} for general spectral curves $C$. The key contribution of our present paper is to determine the conditions under which $U^{\vec{e}}(C)$ is non-empty. In the case $\delta = 0$, this is equivalent to determining when the Hitchin fibration is dominant. (See Remark \ref{hitchin} for an explanation of how to recover the Hitchin correspondence as a special case of Wood's parameterization.)
\end{rem}

\begin{rem}
Theorem \ref{planecurves} about smooth plane curves follows from the case $m = 1$ and $\delta = 0$ of Theorem \ref{maint}.
The case $m = 1$ and $\delta = 1$ also provides information about the Brill--Noether theory of smooth plane curves. Indeed, if $p \in C \subset \pp^2$ is any point, then the projection from $p$ is a degree $k-1$ cover $C \to \pp^1$ that factors through $\mathbb{F}_1$ and the distinguished divisor is $\Delta = p$.
\end{rem}

\begin{cor} \label{thecor}
Suppose $C \subset \mathbb{F}_m$ is a smooth curve of genus $g$ meeting the directrix in degree $\delta$. If $U^{\vec{e}}(C)$ is non-empty, then
\begin{equation} \label{iff} \sum_{i=1}^{k-1} \max\{0, e_{i+1} - e_i - m\} \leq \delta.
\end{equation}
Moreover, if the above condition holds, then $U^{\vec{e}}(C)$ is nonempty for a general curve $C$ meeting the directrix in degree $\delta$.
\end{cor}

We note that if $C \subset \mathbb{F}_m$ meets the directrix in degree $\delta$, then the splitting type of the structure sheaf is 
\begin{equation} \label{so} (-(k-1)m - \delta, -(k-2)m - \delta, \ldots, -m - \delta, 0),
\end{equation}
so \eqref{iff} is sharp.

\subsection{Motivation from Arithmetic Statistics}
Beyond the connection with plane curves, another source of motivation for studying curves on Hirzebruch surfaces comes from arithmetic statistics. The arithmetic analogue of smooth curves on Hirzebruch surfaces are maximal orders obtained by integral binary $k$-ic forms. This is a natural family of number fields that can be worked with explicitly, and thus has been heavily studied in arithmetic statistics. For example, the maximal orders given arising from monic forms are exactly those generated by one element as an algebra; such rings are usually called \emph{monogenic}.  

On the geometric side, we observe that the line bundles on a curve on a Hirzebruch surface behave differently from the line bundles of a general cover of $\PP^1$. Similarly, on the arithmetic side, one can ask if class groups of number fields arising from binary $k$-ic forms behave differently, statistically, from class groups of general number fields. (Of course, one must take into account ordering of the number fields and various other arithmetic difficulties, which we happily ignore in the following discussion). Recent results have shown, amazingly, that the $2$-torsion in the family of number fields arising from binary $k$-ic forms \emph{is} larger on average than the conjectural $2$-torsion in the whole family: see \cite[Theorem 3]{BHS}, \cite[Theorem 1.1]{AS2}, and \cite[Theorem 2]{AS}. This led us to ask if there were notable differences in the Brill-Noether theory of curves on Hirzebruch surfaces as compared to general covers of the projective line. The general question of how the behavior of a family changes when restricted to a special subfamily is an interesting and important one, and this is a natural first case to make this comparison.

Moreover, the splitting type of a line bundle on a cover of $\PP^1$ has been a useful invariant in Hurwitz-Brill-Noether theory \cite{trig, refined,llv,cpj,cpj2}. The splitting type of a line bundle has a natural arithmetic analogue that has played an important role in number theory and dynamics: recall that a fractional ideal $I$ of a degree $k$ number field $L/\QQ$ can be given the structure of a lattice via the Minkowski embedding $L \xhookrightarrow{} \CC^n$. The ``shape'' of this lattice is naturally related to the splitting types of line bundles on curves in the following way. For $1 \leq i \leq k$, let $\lambda_i(I)$ be the smallest positive real number $r$ such that $I$ contains $\geq i$ linearly independent elements of length $\leq r$. Then the $k$-tuple
\[
    (\log \lambda_1(I),\dots,\log \lambda_{k}(I))
\]
should be thought of as analogous to the splitting type $\vec{e}$ of a line bundle on a degree $k$ cover. Let $p_{I}$ be the point in the moduli space of rank $k$ lattices up to isometry corresponding to $I$. When the number field $L$ has a positive rank unit group, it is often fruitful to study the orbit of $p_I$ under the action of the norm $1$ elements of $L_{\RR}^*$; it turns out that these orbits are periodic geodesics on an appropriate Hilbert modular surface. In the case of quadratic extensions and totally real cubic real fields, breakthrough results have shown that these geodesics are equidistributed: see \cite[Theorem 1]{duke} and \cite[Theorem 1.4--1.6]{elmv}, which is analogous to the fact that a general line bundle on a general $k$-cover is balanced. Despite these powerful results, the questions about the periodic geodesics of maximal orders arising from binary $k$-ic remain mysterious. 

Our results suggest that the periodic geodesics arising from fields arising from binary $k$-ic forms have ``different'' behavior from that of a general number field. In particular, one might guess that the lattices are permitted to be more skew and that their cuspidal statistics may differ, since in the geometric setting, our results show that the splitting types of line bundles of curves on Hirzebruch surfaces differ from the splitting types of general covers.

\subsection{Techniques and outline of the paper}
The techniques of this paper are quite different from those in the existing literature on Brill--Noether theory. Typically, proofs in Brill--Noether theory proceed by degenerating to a singular curve, analyzing the limiting behavior, and then using the openness of certain conditions to conclude the general fiber has similar properties. 
After understanding the general fiber,
statements about universal Brill--Noether loci can then be deduced,
possibly after further monodromy arguments.

Here, the situation is reversed. We \emph{start} by building an understanding of the universal Brill--Noether loci $\U^{\vec{e},\vec{f}}$ over the Severi variety $\mathcal{S}$ of smooth curves on Hirzebruch surfaces. 
The main task is then to understand the conditions under which $\U^{\vec{e},\vec{f}}$ dominates $\mathcal{S}$. Once this is determined, no degeneration is required to understand the general fiber: it's dimension is the difference of the dimensions of $\U^{\vec{e},\vec{f}}$ and $\mathcal{S}$. Moreover, in characteristic zero, smoothness of $\U^{\vec{e},\vec{f}}$ implies smoothness of the general fiber.

We now outline the argument in more detail.
Let $\mathcal{S} \subset \pp^N$ be the Severi variety parameterizing smooth, irreducible curves of a fixed linear system on $\mathbb{F}_m$ and let $\mathcal{J} \to \mathcal{S}$ be the relative Picard variety. Each point in $\J$ corresponds to a smooth curve $C \subset \mathbb{F}_m$ together with a line bundle $L$ on $C$.
Let $\U^{\vec{e}, \vec{f}} \subset \J$ be the locally closed subvariety where $L$ has splitting type $\vec{e}$ and $L(\Delta)$ has splitting type $\vec{f}$. By construction, the fiber of $\U^{\vec{e}, \vec{f}} \to \mathcal{S}$ over the point in $\mathcal{S}$ corresponding to $C$ is $U^{\vec{e}, \vec{f}}(C)$.

In Section \ref{wp}, we explain how Wood's parameterization of ideal classes in rings associated to binary $k$-ic forms \cite{Wood} gives rise to a presentation of $\U^{\vec{e}, \vec{f}}$ as a quotient variety.
In our exposition, we give a new perspective on
Wood's parameterization that relates the objects involved to the moduli space of rank $0$, degree $k$ sheaves on $\pp^1$ and results of Str\o mme \cite{Stromme}. 

The key conclusion we extract from Wood's work is that 
$\U^{\vec{e}, \vec{f}} = X^{\vec{e}, \vec{f}}/G^{\vec{e},\vec{f}}$ where $X^{\vec{e}, \vec{f}}$ is an open subset of the affine space $H^0(\pp^1, \Hom(\O(\vec{e}), \O(\vec{f})) \otimes V)$
and 
\[G^{\vec{e},\vec{f}} = (\Aut \O(\vec{e}) \times \Aut\O(\vec{f}))/\gg_m.\]
From this, one readily sees (Corollary \ref{dcor}) that \emph{if it is non-empty}, then $\U^{\vec{e}, \vec{f}}$ is smooth and irreducible of dimension 
\[\dim \U^{\vec{e}, \vec{f}} = \dim \mathcal{S} + g - u(\vec{e}) - u(\vec{f}) + \nu(\vec{e},\vec{f}).\]
However, it is not all all clear when $\U^{\vec{e}, \vec{f}}$ is non-empty! Furthermore, we would like to know under what conditions $\U^{\vec{e}, \vec{f}}$ dominates $\mathcal{S}$.

Fortunately, the map $X^{\vec{e}, \vec{f}} \to \mathcal{S}$ can be made quite explicit. Having fixed a splitting $V \cong \O \oplus \O(m)$, every element of $X^{\vec{e}, \vec{f}}$ can be represented by a pair of $k \times k$ matrices $(A, B)$ whose entries are polynomials on $\pp^1$ of specified degrees 
$\deg A_{ij} = f_j - e_{k + 1 - i}$ and $\deg B_{ij} = f_j - e_{k + 1 - i} + m$. Note that if one of these degrees is zero, we mean the corresponding entry of the matrix vanishes.
When $\vec{e}$ and $\vec{f}$ are more unbalanced, more entries are forced to vanish.
The map $X^{\vec{e},\vec{f}} \to \mathcal{S}$ sends $(A, B)$ to the curve defined by the vanishing of $\det(Ax + By)$. 
If $\vec{e}$ and $\vec{f}$ are too unbalanced, then too many entries are forced to vanish and $\det(Ax + By)$ defines a reducible curve. However, it is possible for some entries to vanish without the curve necessarily being reducible. Proving the dominance of
$X^{\vec{e}, \vec{f}} \to \mathcal{S}$
is rather subtle in these cases and is the topic of Section \ref{dd}.

A standard way to show a map is dominant is by proving the differential is surjective at a well-chosen point. However, locating a suitable point (where the differential is simple enough to understand but not lower rank) is often challenging bordering on impossible! 
Our approach is inspired by this idea, but requires some additional finesse.  One of the key insights is to use the Laplace expansion of the determinant in order to set up an inductive argument.
At certain special points $q$ with many zeros in the bottom row, this allows us to break the tangent space of the source into two pieces:
one that is controlled inductively, and one that is controlled by the bottom row and leftmost column. 
These points $q$ turn out to be too special:
the differential typically drops rank. However, by following a general $1$-parameter family approaching $q$, we find that the flat limits of the images under the differential of the two pieces of our tangent space do span.

In Section \ref{corp}, we prove Corollary \ref{thecor} and Theorem \ref{planecurves}. We conclude with a discussion of splitting types of tensor products of line bundles and a few conjectures in Section \ref{discuss}.

\subsection{Acknowledgements}
We warmly thank Ravi Vakil, Isabel Vogt, and Melanie Wood for valuable conversations related to this work. This research was conducted during the period H.L. served as a Clay Research Fellow. The second author was supported by the National Science Foundation for the duration of this project.

\section{Universal splitting loci via Wood's parameterization}
\label{wp}
\subsection{Notation on the Hirzebruch surface}

Let $V = \O_{\pp^1} \oplus \O_{\pp^1}(m)$ and let $\mathbb{F}_m = \pp V$. Let $\gamma: \pp V \to \pp^1$ be the projection. Let $H = \O_{\pp V}(1)$ be the relative hyperplane class and let $F = \gamma^*\O_{\pp^1}(1)$ be the class of a fiber. The intersection numbers are:
\[H^2 = m, \qquad H \cdot F = 1, \qquad F^2 = 0.\]
The class of the directrix is $D := H - mF$. 
The canonical class is $K_{\mathbb{F}_m} = -2H + (m- 2)F$, and the relative canonical class is
\begin{equation} \label{rc} \omega_\gamma = -H - D.
\end{equation}

For example, if $m = 1$, then $\mathbb{F}_1$ is the blow up of $\pp^2$ at a point $p$. With the notation we have developed, $H$ is the pullback of the class of a line on $\pp^2$, $F$ is the class of the strict transform of a line through $p$, and $D$ is the class of the exceptional divisor of the blow up.

Suppose $\iota: C \hookrightarrow \pp V$ is a curve of class $k H + \delta F$. By adjunction, we have
\[K_C =(K_{\mathbb{F}_m} + C)\vert_C = \iota^*((k - 2) H + (\delta + m - 2) F).\]
Calculating its degree gives the following formula for the genus of $C$
\begin{equation} \label{gf} g(C) = {k \choose 2} m + (k - 1)(\delta - 1).
\end{equation}
Let $\alpha = \gamma \circ \iota: C \to \pp^1$.
Note also that
\begin{equation} \label{oa}
\omega_{\alpha} = K_C \otimes \alpha^*\omega_{\pp^1}^\vee = \iota^*((k - 2)H + (\delta + m)F) = \iota^* \O_{\pp V}(C-D - H).
\end{equation}

\subsection{The moduli space of curves on Hirzebruch surfaces}
Fix nonnegative integers $m, k$ and $\delta$, and let 
\[Y = Y_{m, k, \delta} \subset H^0(\mathbb{F}_m, \O_{\mathbb{F}_m}(kH + \delta F)) = H^0(\pp^1, \Sym^k V \otimes \O_{\pp^1}(\delta))\]
be the open subset of binary forms whose vanishing defines a smooth, irreducible curve inside $\mathbb{F}_m$.
Explicitly, we can write an element of $H^0(\pp^1, \Sym^k V \otimes \O_{\pp^1}(\delta))$ as
\[P(x, y) = P_k(t)x^k + \ldots + P_1(t) y^{k-1}x + 
P_0(t)y^k \]
where $P_i(t) \in H^0(\pp^1, \O_{\pp^1}(\delta + (k-i)m))$ is the coefficient of $x^i$. (For ease of notation, we write the inhomogeneous form $P_i(t)$ instead of the homogeneous form $P_i(s, t)$ for polynomials on $\pp^1$.)
The directrix of $\pp V$ is defined by the vanishing of $y$.

The projectivization of $Y$ is the Severi variety $\mathcal{S} \subset \pp  H^0(\mathbb{F}_m, \O_{\mathbb{F}_m}(kH + \delta F))$ of smooth curves in the given linear system. The universal family of curves over $\mathcal{S}$ sits in a diagram
\begin{equation} \label{uc}
\begin{tikzcd}
\mathcal{C} \arrow{dr} \arrow[hook]{r} & \pp V \times \mathcal{S} \arrow{d} \\
& \mathcal{S}.
\end{tikzcd}
\end{equation}

\subsection{The Picard variety} \label{pv}
Let $\mathcal{J} \to \mathcal{S}$ be the relative Picard variety of this family of curves. 
Each geometric point in $\J$ corresponds to a pair $(C, L)$ where
$C \subset \mathbb{F}_m$ is a smooth curve of class $kH + \delta F$ and $L$ is a line bundle on $C$.  The universal splitting locus $\U^{\vec{e}, \vec{f}} \subset \J$ is the locally closed subvariety of points $(C, L)$ where $\alpha_*L \cong \O(\vec{e})$ and $\alpha_*L(\Delta) \cong \O(\vec{f})$.

\subsection{The relative Picard stack}
Typically, there is not a universal line bundle on the universal curve $\mathcal{C} \times_{\mathcal{S}} \J$ over the relative Picard variety $\J$. 
It is thus useful to also introduce the relative Picard \emph{stack} $\mathscr{J} \to \mathcal{S}$. An object of the stack $\mathscr{J}$ over a scheme $S$ is a pair $(C, L)$
where $C \hookrightarrow \pp V \times S$ is a closed subvariety such that $C \to S$ is a family of smooth curves of class $k H + \delta F$ in each fiber, and $L$ is a line bundle on $C$.  A morphism $(C, L) \to (C', L')$ is an equality of closed subschemes $C = C'$ together with an isomorphism $L \cong L'$.

Let $\mathscr{L}$ be the universal line bundle on $\mathcal{C} \times_{\mathcal{S}} \mathscr{J}$ and let $\alpha: \mathcal{C} \times_{\mathcal{S}} \mathscr{J} \to \pp^1 \times \mathscr{J}$ be the natural degree $k$ map. We have rank $k$ vector bundles $\mathscr{E} := \alpha_*\mathscr{L}$ and $\mathscr{F} := \alpha_*\mathscr{L}$ on $\pp^1 \times \mathscr{J}$.
We define the locally closed substack $\mathscr{U}^{\vec{e},\vec{f}} \subset \mathscr{J}$ to be the intersection of the $\vec{e}$ splitting locus of $\mathscr{E}$ with the $\vec{f}$ splitting locus of $\mathscr{F}$.
We have a natural fiber diagram
\begin{center}
\begin{tikzcd}
\mathscr{U}^{\vec{e},\vec{f}} \arrow{r} \arrow{d} &  \mathscr{J} \arrow{d} \\
\U^{\vec{e},\vec{f}}  \arrow{r} & \J,
\end{tikzcd}
\end{center}
where the vertical maps are $\gg_m$-gerbes.

\subsection{Matrices of linear forms}
Fix splitting types $\vec{e}$ and $\vec{f}$ with $\deg(\vec{f}) - \deg(\vec{e}) = \delta$. Consider the affine space 
\begin{equation} \label{equ} H^0(\pp^1, \Hom(\O(\vec{e}), \O(\vec{f})) \otimes V) = H^0(\pp V, \gamma^*\Hom(\O(\vec{e}), \O(\vec{f})) \otimes \O_{\pp V}(1)).
\end{equation}
Having fixed a splitting $V = \O \oplus \O(m)$, we can represent an element of the left-hand side by a pair
\[(A, B) \in H^0(\pp^1, \Hom(\O(\vec{e}), \O(\vec{f}))) \oplus H^0(\pp^1, \Hom(\O(\vec{e}), \O(\vec{f}))(m)). \]
Then, we represent each of $A$ and $B$ with $k\times k$ matrices whose entries $A_{ij}$ and $B_{ij}$ are polynomials of a specified degree on $\pp^1$:
\begin{align} A_{ij} \in H^0(\pp^1, \O_{\pp^1}(a_{i,j})) \qquad &\text{and} \qquad B_{ij} \in H^0(\pp^1, \O_{\pp^1}(b_{i,j})),
\intertext{where}
a_{i,j} = f_i - e_{k + 1 - j} \qquad &\text{and} \qquad b_{i,j} = f_i - e_{k + 1 - j} +m. \label{aij}
\end{align}
Notice our choice of indexing is so that $a_{i,j}$ and $b_{i,j}$ are increasing as $i$ and $j$ increase.

When we think of $(A, B)$ as an element of the right-hand side of \eqref{equ}, we write it as $Ax + By$. 
There is a natural map
\begin{equation} \label{phidef} \phi^k: H^0(\pp^1, \Hom(\O(\vec{e}), \O(\vec{f})) \otimes V)  \to H^0(\pp^1, \Sym^k V \otimes \O_{\pp^1}(\delta))
\end{equation}
given by
\[(A, B) \mapsto \det(Ax + By)\]
We define $X^{\vec{e}, \vec{f}} \subset H^0(\pp^1, \Hom(\O(\vec{e}), \O(\vec{f})) \otimes V)$ to be the preimage of $Y$ under $\phi^k$.

Notice that there is a natural action of $\Aut \O(\vec{e}) \times \Aut \O(\vec{f})$ on $X^{\vec{e}, \vec{f}}$ given by 
\[(g_1, g_2) \cdot (A, B) = (g_2 A g_1^{-1}, g_2 B g_1^{-1}).\]
The diagonal subgroup $\gg_m \hookrightarrow \Aut \O(\vec{e}) \times \Aut \O(\vec{f})$ given by $t \mapsto (tI, tI)$ clearly acts trivially. Let $G^{\vec{e}, \vec{f}} = (\Aut \O(\vec{e}) \times \Aut \O(\vec{f}))/\gg_m$ be the quotient by this subgroup.
The  map $X^{\vec{e}, \vec{f}} \to \mathcal{S}$ is equivariant with respect to the $G^{\vec{e},\vec{f}}$-action on $X^{\vec{e},\vec{f}}$, and therefore factors through the quotient $X^{\vec{e}, \vec{f}}/G^{\vec{e},\vec{f}}$. We will show that this quotient is in fact $\U^{\vec{e},\vec{f}}$, so that we have a commuting diagram 
\begin{equation} \label{com}
\begin{tikzcd}
X^{\vec{e}, \vec{f}} \arrow{d} \arrow{r} & \U^{\vec{e}, \vec{f}} \arrow{d} \\
Y \arrow{r} & \S,
\end{tikzcd}
\end{equation}
where the top row is a $G^{\vec{e},\vec{f}}$-torsor. (Note that the bottom row is a $\gg_m$-torsor.)

This is essentially a consequence of Wood's parameterization of ideal classes in rings associated to binary forms \cite[Theorem 1.4]{Wood}, which holds over any base.
For the sake of completeness, we give a self-contained proof here in the next subsection. 

\subsection{Str{\o}mme sequences} \label{ss}
In this section, we give a new perspective on Wood's parameterization by relating the structures involved to the moduli space of rank $0$ degree $k$ subsheaves on $\pp^1$ and results of Str\o mme.
Our goal is to establish the diagram \eqref{com}.
In fact, we will show that there is an equivalence of stacks
\[\mathscr{U}^{\vec{e},\vec{f}} \cong [X^{\vec{e},\vec{f}}/\Aut \O(\vec{e}) \times \Aut \O(\vec{f})].\]
This then implies that that their $\gg_m$-rigidifications $\U^{\vec{e},\vec{f}}$ and $X^{\vec{e},\vec{f}}/G^{\vec{e},\vec{f}}$ are equivalent.

Given an object $(C, L)$ of the relative Picard stack $\mathscr{J}$ over a scheme $S$, we have a diagram 
\begin{center}
\begin{tikzcd}
C \arrow[hook]{r}{\iota} \arrow{dr}[swap]{\alpha} & \pp V \times S \arrow{d}{\gamma} \\
& \pp^1 \times S.
\end{tikzcd}
\end{center}
Let $M = L(\Delta)$ and consider the sheaf $\iota_*M$ on $\pp V \times S$. 
Since $R^1\gamma_*(\iota_*M) = 0$, by \cite[Proposition 1.1]{Stromme}, there is a short exact sequence of sheaves on $\pp V \times S$
\begin{equation} \label{Stromme} 0 \rightarrow \gamma^*(\gamma_*(\iota_*M \otimes \O_{\pp V}(-1)) \otimes \det V)(-1) \rightarrow \gamma^*(\gamma_* \iota_* M) \rightarrow \iota_* M \rightarrow 0.
\end{equation}
This sequence is functorial in $\iota_*M$ and its formation commutes with arbitrary base change. We call it the ``Str\o mme sequence of $\iota_*M$."
\begin{rem} The original \cite[Proposition 1.1]{Stromme} is stated for trivial $\pp^1$-bundles. For non-trivial $\pp^1$-bundles, we must include a factor of the determinant in the left-hand term as in \cite[Lemma 3.1]{degen}. (To account for the missing dual, note that \cite{degen} used the subspace convention whereas we are using the quotient convention with regards to $\pp V$.)
\end{rem}

Let us set
\begin{align*} 
E &:= \gamma_*(\iota_*M \otimes \O_{\pp V}(-1)) \otimes \det V = \gamma_*\iota_*(M \otimes \iota^*(\O_{\pp V}(-1) \otimes \gamma^* \det V)) \\
&= \alpha_*M(-\Delta) = \alpha_*L \\
\intertext{and}
F &:= \gamma_* \iota_*M = \alpha_* M
\end{align*}
so that \eqref{Stromme} becomes
\begin{equation} \label{us}
0 \rightarrow (\gamma^*E)(-1) \rightarrow \gamma^* F \rightarrow \iota_* M \rightarrow 0.    
\end{equation}

\begin{rem}
On each fiber of $\gamma$, the bundle $F$ is the space of global sections of $\iota_*M$ restricted to that fiber. One can therefore think of $(\gamma^* E)(-1)$ as a sort of ``relative version" of a Lazarsfeld--Mukai bundle.
\end{rem}

The map $(\gamma^*E)(-1) \to \gamma^* F$ in \eqref{us} can be represented by a family of matrices of linear forms
\[Ax + By \in H^0(\pp V \times S, \gamma^*\Hom(E, F) \otimes \O_{\pp V}(1)).\]
The vanishing of 
\[\det(Ax + By) \in H^0(\pp V \times S, \gamma^*((\det E)^\vee \otimes \det F) \otimes \O_{\pp V}(k))\]
is the support of $\iota_* M$, namely the curve $C$. Notice that $(\det E)^\vee \otimes \det F$ must have degree $\delta$ on each fiber of $\pp^1 \times S \to S$. This can also be seen directly from the definitions of $E$ and $F$, using Riemann--Roch on each fiber as in Remark \ref{nec}(3).

The association of $(C, L)$ with $(E, F, Ax + By)$ gives rise to a natural equivalence of $\mathscr{J}$ with a stack $\mathscr{J}'$, which we now define.
An object of $\mathscr{J}'$ over a scheme $S$ is a triple $(E, F, Ax + By)$ where
\begin{enumerate}
    \item $E$ and $F$ are vector bundles of rank $k$ on $\pp^1 \times S$ such that $(\det E)^\vee \otimes \det F$ has degree $\delta$ on the fibers of $\pp^1 \times S \to S$
    \item $Ax + By$ is a global section
    \begin{align*} Ax + By &\in H^0(\pp V \times S, \gamma^*\Hom( E, F) \otimes \O_{\pp V}(1)))
    \end{align*}
    such that the projection of $V(\det(Ax + By)) \subset S \times \pp V$ to $S$ is a family of smooth curves in $\pp V$ (which necessarily have class $kH + \delta F$ in each fiber by (1)).
\end{enumerate}
A morphism $(E, F, Ax+ By) \to (E', F', A'x+ B'y)$ of objects in $\mathscr{J}'$ over $S$ is the data of isomorphisms of vector bundles $\xi_1 : E \to E'$ and $\xi_2: F \to F'$ such that the following diagram commutes
\begin{center}
\begin{tikzcd}
(\gamma^*E) \otimes \O_{\pp V}(-1) \arrow{rr}{Ax + By} \arrow{d}[swap]{\gamma^*\xi_1 \otimes \mathrm{id}} && \gamma^* F \arrow{d}{\gamma^*\xi_2} \\
(\gamma^*E') \otimes \O_{\pp V}(-1) \arrow{rr}[swap]{A'x + B'y} && \gamma^* F'.
\end{tikzcd}
\end{center}
Functoriality of the Str\o mme sequence implies there is a well-defined functor from $\mathscr{J}$ to $\mathscr{J}'$ sending $(C, L)$
to $(E, F, Ax + By)$. The next lemma helps us towards defining an inverse.

\begin{lem} \label{ml}
Suppose $(E, F, Ax + By)$ is an object of $\mathscr{J}'$. Let 
\[C = V(\det(Ax + By)) \xhookrightarrow{\iota} S \times \pp V.\]
Then $Ax + By$ has rank exactly $k - 1$ at each point of $C$. Hence, there is some line bundle $M$ on $C$ such that
\begin{equation} \label{mss} 0 \rightarrow (\gamma^* E)(-1) \xrightarrow{Ax + By} \gamma^* F \rightarrow \iota_*M \rightarrow 0
\end{equation}
is exact and is the Str\o mme sequence of $\iota_*M$. In particular $(C, M(-\Delta))$ defines an object of $\mathscr{J}$ which is sent to $(E, F, Ax + By)$.
\end{lem}
\begin{proof}
The matrix $Ax + By$ defines a map of $\pp V \times S$ to the projective bundle $\pp \Hom(E, F)$ over $S$. The fiber of this projective bundle over each point in $S$ is a projective space $\pp^{k^2 - 1}$ of $k \times k$ matrices.
The subvariety $C$ is the preimage of the determinant hypersurface inside this projective bundle. It is well-known that the determinant hypersurface is singular along the locus of matrices having rank $< k - 1$, i.e. the tangent space to the hypersurfaces at such a point is all of $\pp^{k^2 - 1}$. If $\pp V$ meets the locus of matrices  having rank $< k - 1$, then the tangent space to the fiber of $C \to S$ at such a point $p$ would be all of $\pp V$, contradicting the fact that $C \to S$ is a family of smooth curves. Thus, $Ax + By$ has rank exactly $k - 1$ at each point of $C$, and we conclude that the cokernel is $\iota_*M$ for a locally free sheaf $M$ on $C$.

To see that \eqref{mss} is the Str\o mme sequence of $\iota_*M$, we apply $\gamma_*$. 
By cohomology and base change, $\gamma_*(\gamma^*E)(-1) = R^1\gamma_*(\gamma^*E)(-1) = 0$. Thus,
when we pushforward \eqref{mss} along $\gamma$, we obtain an isomorphism $F \cong \gamma_*\iota_*M = \alpha_*M$, where $\alpha = \gamma \circ \iota: C \to \pp^1 \times S$ is the natural degree $k$ cover.
To identify $E$, we take the tensor product of \eqref{mss} with $\O_{\pp V}(-D)$ and push forward by $\gamma_*$. 
The pushforward and higher pushforward of the middle term in the resulting sequence vanishes,  so we obtain an isomorphism of $\alpha_* M(-\Delta) = \gamma_*\iota_*(M(-\Delta)) = \gamma_*(\iota_*M)(-D)$ with 
\[R^1\gamma_*(\gamma^* E)(-H - D) = E \otimes R^1\gamma_* \omega_{\gamma} = E.\]
Above, the first equality follows from the projection formula and \eqref{rc}, while the second equality follows from relative Serre duality.
\end{proof}

\begin{lem} \label{eq}
There is an equivalence of stacks $\mathscr{J} \cong \mathscr{J}'$. 
\end{lem}
\begin{proof}
In light of Lemma \ref{ml}, we can define a functor from $\mathscr{J}'$ to $\mathscr{J}$ that sends the triple $(E, F, Ax + By)$ to $(C = V(\det(Ax + By)), L = \iota^* \coker(Ax + By)(-\Delta))$. The composition $\mathscr{J} \to \mathscr{J}' \to \mathscr{J}$ first sends $(C, L)$ to the Str\o mme sequence of $\iota_*L(\Delta)$ and then back to the support and cokernel of the sequence twisted down by $\Delta$, so is evidently the identity. Meanwhile, the composition $\mathscr{J}' \to \mathscr{J} \to \mathscr{J}'$ sends a Str\o mme sequence to data equivalent to its cokernel, and then recovers the Str\o mme sequence from it.
\end{proof}

\begin{rem}[Translation with Wood]
If $L$ is a line bundle on $C \subset \pp V \times S$, then the pair of line bundles $(L, L^\vee \otimes \omega_{\alpha}(-\Delta))$ is what Wood calls a ``balanced pair" of $\O_C$-modules.
To translate between our notation and Wood's, we first explain why Wood's line bundle $\O_{T_f}(n - 3) \otimes (\wedge^2 V)^{\otimes 2} \otimes L$ in \cite[p. 187]{Wood} ($I_f$ viewed as an $R_f$-module)  is isomorphic to $\omega_{\alpha}(-\Delta)$. This is because Wood's $\O_{T_f}(n) \otimes L$ is our $\O_C(C)$, while $\wedge^2V = \O(mF)$ and $\O_{T_f}(-3) = \O(-3H)$. Note also that $\omega_{\gamma} = \O(-2H + mF)$, and by adjunction, $\omega_\alpha = \omega_\gamma(C)$.
Thus, the line bundle that pushes forward to Wood's $I_f$ is 
\[\O_C(C - 3H + 2mF) = \omega_\gamma(C) \otimes \O(-H + mF) = \omega_\alpha \otimes \O(-\Delta).\]
Hence the tensor product of the pair $(L, L^\vee \otimes \omega_{\alpha}(-\Delta))$ pushes forward to $I_f$.
Viewed as $\O_{\pp^1}$-modules, these become $\alpha_*L = E$ and 
\[\alpha_*(L^\vee \otimes \omega_{\alpha}(-\Delta)) = \alpha_*(\omega_{\alpha} \otimes L(\Delta)^\vee) = (\alpha_* L(\Delta))^\vee = F^\vee.\]
Wood's ``balancing map" is a map $E \otimes F^\vee \to V$, which is equivalent to the data of the map $\gamma^*E \otimes \O_{\pp V}(-1) \to \gamma^* F$ in the Str\o mme sequence.
\end{rem}

\begin{rem} \label{hitchin}
The special case $\delta = 0$ of Wood's parameterization is closely related to spectral curves.
We explain the situation over the base curve $\pp^1$, but Wood's parameterization is valid over any base, and replacing $\pp^1$ with an arbitrary curve $X$ recovers
\cite[Proposition 3.6]{BNR} about the moduli space of torsion-free sheaves on spectral curves.

When $\delta = 0$, our curves do not meet the directrix of the Hirzebruch surface, and so
are contained in the total space of the line bundle $\O(m)$ over $\pp^1$.
Our family of curves in $S \times \pp V$ is given by the vanishing of $\det(Ax + By)$. The coefficient of $x^k$ is $\det A$, so to obtain an integral curve in each fiber, we must have $\det A \neq 0$.
When $\delta = 0$, we have $\det A \in H^0(S \times \pp^1, \O)$ is non-vanishing. Geometrically, this tells us that the curve does not meet the directrix (which is the curve defined by $y = 0$). 
In particular, we can view the matrix $A$ as providing an isomorphism of $E$ with $F$.
Then, $B \in H^0(\pp^1, \End(E)(m))$ gives a family of endomorphisms of $E$. Working on the complement of the directrix $y \neq 0$, we can define a relative coordinate $t = -x/y$ on the total space of $\O(m)$ and our curve is then given
by the vanishing of 
$\det(t \mathrm{Id} - B)$. In other words, the curve is the spectral curve of the family of endomorphisms defined by $B$.
\end{rem}

\subsection{Conclusion for splitting loci}
We now restrict the equivalence of stacks in Lemma \ref{eq} to the substack $\mathscr{U}^{\vec{e},\vec{f}}$ to obtain the following description.

\begin{lem}
There is an equivalence of stacks $\mathscr{U}^{\vec{e},\vec{f}} \cong [X^{\vec{e},\vec{f}}/\Aut \O(\vec{e}) \times \Aut \O(\vec{f})]$. Consequently, there is an isomorphism of the $\gg_m$-rigidifications $\U^{\vec{e},\vec{f}} \cong X^{\vec{e},\vec{f}}/G^{\vec{e},\vec{f}}$.
\end{lem}
\begin{proof}
By definition, $X^{\vec{e},\vec{f}} \subset H^0(\pp^1, \Hom(\O(\vec{e}), \O(\vec{f})) \otimes V)$ is the open subset consisting of $(A, B)$ such that $V(\det(Ax + By))$ is a smooth curve. Thus, the quotient stack $[X^{\vec{e},\vec{f}}/\Aut \O(\vec{e}) \times \Aut \O(\vec{f})]$ is precisely the stack of triples $(E, F, Ax + By) \in \mathscr{J}'$ where $E$ and $F$ have splitting type $\vec{e}$ and $\vec{f}$ on every fiber of $\pp^1 \times S \to S$. The equivalence of stacks in Lemma \ref{eq} gives an equivalence of this substack of
$\mathscr{J}'$ with $\mathscr{U}^{\vec{e},\vec{f}} \subset \mathscr{J}$.

For the second claim, note that the subgroup $\mathbb{G}_m \hookrightarrow \Aut \O(\vec{e}) \times \Aut \O(\vec{f})$ defined by $t \mapsto (tI, tI)$ corresponds to scaling the fibers of $E$ and $F$ equally. This in turn corresponds to scaling the fibers of $L$, and accounts for a $\gg_m$ stabilizer at each point of $\mathscr{U}^{\vec{e},\vec{f}}$. 
The variety $\mathcal{U}^{\vec{e},\vec{f}}$ is the rigidification of $\mathscr{U}^{\vec{e},\vec{f}}$ with respect to this $\gg_m$ action. As $G^{\vec{e},\vec{f}}$ is the quotient of $\Aut \O(\vec{e}) \times \Aut \O(\vec{f})$ by this copy of $\gg_m$, the claim for $\U^{\vec{e},\vec{f}}$ follows.
\end{proof}

\begin{cor} \label{dcor}
If non-empty, the variety $\U^{\vec{e}, \vec{f}}$ is smooth and irreducible of dimension
\[\dim \U^{\vec{e}, \vec{f}} = \dim \S + g - u(\vec{e}) - u(\vec{f}) + \nu(\vec{e},\vec{f}). \]
\end{cor}
\begin{proof}
Assume that $X^{\vec{e},\vec{f}}$ is non-empty. Then, its dimension is
\begin{align*} \dim X^{\vec{e},\vec{f}} &= h^0(\pp^1, \Hom(\O(\vec{e},\O(\vec{f})) \otimes V) \\
& = \chi (\pp^1, \Hom(\O(\vec{e}),\O(\vec{f})) \otimes V) + \nu(\vec{e},\vec{f}) \\
&= 2k^2 + \sum_{1 \leq i,j \leq k} f_j - e_i + \sum_{1 \leq i,j \leq k} (f_j - e_i +m) + \nu(\vec{e},\vec{f}) \\
&= 2k^2 + 2k(\deg(\vec{f}) - \deg(\vec{e})) + 2km + \nu(\vec{e},\vec{f}) \\
&= 2k^2 + 2k \delta + k^2 m + \nu(\vec{e},\vec{f}).
\end{align*}
Next, we have
\begin{align*} 
\dim G^{\vec{e},\vec{f}} &= h^0(\pp^1, \End \O(\vec{e})) + h^0(\pp^1, \End \O(\vec{f})) - 1 \\
&= 2k^2 + u(\vec{e}) + u(\vec{f}) - 1.
\end{align*}
Thus, we have
\begin{equation} \label{d3} \dim \U^{\vec{e},\vec{f}} = \dim X^{\vec{e},\vec{f}} - \dim G^{\vec{e},\vec{f}} = 2k\delta + k^2m + 1  - u(\vec{e}) - u(\vec{f}) + \nu(\vec{e},\vec{f}).
\end{equation}
Finally, we compute
\begin{align*}
\dim \mathcal{S} + g &= h^0(\pp^1, \Sym^{k} V \otimes \O_{\pp^1}(\delta))  - 1 + g\\
&= k + (0 + m + 2m + \ldots + km) + (k + 1)\delta + g\\
&= k + \frac{1}{2}(k + 1)km + (k + 1)\delta + \frac{1}{2}k(k - 1)m + (k - 1)(\delta - 1) \\
&=2k\delta +  k^2 m +1.
\end{align*}
Substituting this into \eqref{d3}, the dimension claim now follows. Since $X^{\vec{e},\vec{f}}$ is smooth and irreducible and $X^{\vec{e},\vec{f}} \to \U^{\vec{e},\vec{f}}$ is a $G^{\vec{e},\vec{f}}$ torsor, $\U^{\vec{e},\vec{f}}$ is also smooth and irreducible.
\end{proof}

\begin{rem}
Notice that the universal splitting locus $\mathcal{U}^{\vec{e},\vec{f}}$ is irreducible. 
 Thus, although the general fiber of $\U^{\vec{e},\vec{f}} \to \S$ may be reducible, all components must be exchanged under monodromy (see Example \ref{trige}).
\end{rem}

\section{Dominance of the determinant map and proof of Theorem \ref{maint}} \label{dd}
In the previous section, we described $\mathcal{U}^{\vec{e},\vec{f}}$ as a quotient of an open subset of an affine space. Our next task is to understand the geometry of a general fiber
$\mathcal{U}^{\vec{e},\vec{f}} \to \mathcal{S}$. In particular, we would like to know when $\mathcal{U}^{\vec{e},\vec{f}} \to \mathcal{S}$ is dominant. By the commutative diagram \eqref{com}, we know that $\mathcal{U}^{\vec{e},\vec{f}} \to \mathcal{S}$ is dominant if and only if $X^{\vec{e},\vec{f}} \to Y$ is dominant.

Fix splitting types $\vec{e}$ and $\vec{f}$ with $\deg(\vec{f}) - \deg(\vec{e}) = \delta$
and let $a_{i,j}$ and $b_{i,j}$ be as in \eqref{aij}.
Let $\phi^k$ be the map of affine spaces defined in \eqref{phidef}. Recall that $Y \subset H^0(\pp^1, \Sym^k V \otimes \O_{\pp^1}(\delta))$ is the locus of equations defining smooth, irreducible curves and
$X^{\vec{e},\vec{f}} = (\phi^k)^{-1}(Y)$.
We first observe some conditions that easily imply that $X^{\vec{e},\vec{f}}$ is empty.

\begin{lem} \label{easy}
If $a_{i,k - i+ 1} < 0$ for some $i$ or $b_{i,k - i} < 0$ for some $i$, then the image of $\phi^k$ is contained in the complement of $Y$. Consequently, $X^{\vec{e},\vec{f}}$ is empty.
\end{lem}
\begin{proof}
Recall that the $a_{i,j}$ are increasing with $i$ and $j$.
Thus, if $a_{i,k-i+1} < 0$ for some $1 \leq i \leq k$, then the first $i$ rows of $A$ are supported in the last $i-1$ columns. This implies that $\det A = 0$, so $\det(Ax + By)$ is divisible by $y$ for every $(A, B)$. The vanishing of $\det(Ax + By)$ is therefore a reducible curve, so it lies in the complement of $Y$.

Now suppose $b_{i,k - i} < 0$ for some $1 \leq i \leq k-1$. Note that $a_{i,j} \leq b_{i,j}$ by \eqref{aij}. It follows that $Ax + By$ has a block shape where the first $i$ rows are supported in the last $i$ columns. Thus, the determinant of the $i \times i$ block formed by the first $i$ rows and the last $i$ columns divides $\det(Ax + By)$. In particular, the vanishing of $\det(Ax + By)$ defines a reducible curve, so lies in the complement of $Y$.
\end{proof}

Lemma \ref{easy} shows that, in order for $X^{\vec{e},\vec{f}} \to Y$ to be dominant, it is certainly necessary that $a_{i,k - i+ 1} \geq 0$ and $b_{i,k - i} \geq 0$ for all $i$.
The main challenge will be to show that these conditions are also sufficient.

\begin{lem} \label{mainl}
If $a_{i,k - i+ 1} \geq 0$ and $b_{i,k - i} \geq 0$, then $\phi^k$ is dominant.
\end{lem}

\begin{proof}[Proof of Theorem \ref{maint} assuming Lemma \ref{mainl}]
We first rewrite conditions (1) and (2) in the statement of Theorem \ref{maint} in terms of $a_{i,j}$ and $b_{i,j}$.
By \eqref{aij}, we have $a_{i,k-i+1} = f_i - e_{i}$, so condition (1) in Theorem \ref{maint} is equivalent to $a_{i,k-i+1} \geq 0$. Also by \eqref{aij}, we have $b_{i, k - i} = f_i - e_{i + 1} + m$, so condition (2) in Theorem \ref{maint} is equivalent to $b_{i,k-i} \geq 0$. Note that condition (3) in Theorem \ref{maint} is equivalent to $\deg(\vec{f}) - \deg(\vec{e}) = \delta$.

If either of conditions (1) or (2) fails, then Lemma \ref{easy} says that $X^{\vec{e},\vec{f}}$, and hence $\U^{\vec{e},\vec{f}}$, is empty. 
Note that condition (3) is also necessary for degree reasons (see Remark \ref{nec}).
This proves the first claim in Theorem \ref{maint}.

Meanwhile, if conditions (1) -- (3) hold, then Lemma \ref{mainl} says that 
\[\phi^k: H^0(\pp^1, \Hom(\O(\vec{e}), \O(\vec{f})) \otimes V) \to H^0(\pp^1, \Sym^k V \otimes \O_{\pp^1}(\delta))\] is dominant. Since $Y \subset H^0(\pp^1, \Sym^k V \otimes \O_{\pp^1}(\delta))$ is open, the map from its preimage $X^{\vec{e},\vec{f}} \to Y$ is also dominant.
By \eqref{com}, it follows that $\U^{\vec{e},\vec{f}} \to \mathcal{S}$ is dominant.
Thus, the dimension of the general fiber is
\[\dim U^{\vec{e},\vec{f}}(C) = \dim \U^{\vec{e},\vec{f}} - \dim \mathcal{S} = g - u(\vec{e}) - u(\vec{f}) + \nu(\vec{e}, \vec{f}), \]
where the second equality follows from Corollary \ref{dcor}.

By Corollary \ref{dcor}, we also know that $\U^{\vec{e},\vec{f}}$ is smooth. If the characteristic of the ground field is zero, then generic smoothness implies that that the general fiber is also smooth.
\end{proof}

The remainder of this section is devoted to proving Lemma \ref{mainl}.
Notice that the condition $a_{i, k -i + 1} \geq 0$ says that the degrees of all entries of $A$ on or below the anti-diagonal are nonnegative. Similarly, $b_{i, k - i } \geq 0$ implies that the degrees of all entries of $B$ on or below the line one above the anti-diagonal are nonnegative.
In what follows, the phrases ``upper triangular" and ``lower triangular" are meant with respect to the \emph{anti}-diagonal.

\subsection{Motivating example and base case: $k = 2$} \label{k2}
As motivation, we present a simple argument that $\phi^k$ is dominant (in fact surjective) when $k = 2$ and the conditions of Lemma \ref{mainl} are satisfied. Let $\mathsf{LU} \subset H^0(\pp^1, \Hom(\O(\vec{e}), \O(\vec{f})) \otimes V)$ be the 
linear subspace defined by $A_{11} = 0$ and $B_{22} = 0$. This is the subspace where $A$ is $\mathsf{L}$ower triangular and $B$ is $\mathsf{U}$pper triangular. For $(A, B) \in \mathsf{LU}$, we have
\[Ax + By = \left(\begin{matrix}
B_{11} y & A_{12}x + B_{12}y \\
A_{21}x + B_{21}y & A_{22}x
\end{matrix} \right)\]
so $\phi^2$ takes a simpler form, namely
\begin{equation} \label{sf} \phi^2(A, B) = -A_{21}A_{12}x^2 - (A_{12}B_{21} + A_{21}B_{12} - B_{11}A_{22})xy - B_{21}B_{12}y^2.
\end{equation}
It is now easy to see that $\phi^k$ is dominant by inspection. We may choose $A_{21}, A_{12}, B_{21}, B_{12}$ to obtain any desired coefficient of $x^2$ and $y^2$. After that, we still have the freedom to make
$B_{11}A_{22}$ arbitrary, so we can get any desired coefficient of $xy$.

As an example of Lemma \ref{easy}, we note that if either $a_{21} < 0$ or $a_{12} < 0$, forcing $A_{21}$ or $A_{12}$ to vanish, then the coefficient of $x^2$ would vanish. Additionally, if we had $b_{11} < 0$, forcing $B_{11} = 0$, then the resulting curve would be reducible.

\subsection{The subspaces $\mathsf{LU}$ and $\SUT$}
More generally, for any $k$, let us define \[\mathsf{LU} \subset H^0(\pp^1, \Hom(\O(\vec{e}), \O(\vec{f})) \otimes V)\]
to be the linear subspace where $A_{i,j} = 0$ for $i + j < k + 1$ and $B_{i,j} = 0$ for $i + j > k + 1$; in other words, where $A$ is $\mathsf{L}$ower triangular and $B$ is $\mathsf{U}$pper triangular.
It suffices to show that the restriction of $\phi^k$ to $\mathsf{LU} \to H^0(\pp^1, \Sym^k V \otimes \O_{\pp^1}(\delta))$ is dominant.

Given $P(x, y) \in H^0(\pp^1, \Sym^k V \otimes \O_{\pp^1}(\delta))$, we write
\[P(x, y) = P_k(t)x^k + \ldots + P_1(t) y^{k-1}x + 
P_0(t)y^k \]
where $P_i(t) \in H^0(\pp^1, \O_{\pp^1}(\delta + (k-i)m))$ is the coefficient of $x^i$.
Recall that the directrix in $\pp V$ is defined by the vanishing of $y$. 

To motivate what comes next, notice that for $(A, B) \in \mathsf{LU}$, the coefficient $P_0(t)$ depends only on the anti-diagonal entries of $B$. Explicitly, writing $\diag(B) = (B_{k,1}, \ldots, B_{1,k})$ for the tuple of anti-diagonal entries, 
we can factor $\phi^k$ as
\begin{equation} \label{factor} (A, B) \mapsto (\diag(B), P_1(t), \ldots, P_{k}(t)) \mapsto (\prod \diag(B), P_1(t), \ldots, P_{k}(t)).
\end{equation}
The latter map from $\{\diag(B), P_1(t), \ldots, P_{k}(t)\} \to \{P_0(t), P_1(t), \ldots, P_{k}(t)\}$ is clearly surjective, since all entries on the anti-diagonal of $B$ have nonnegative degree.
It therefore suffices to show that the first map above is dominant. Notice that this map commutes with projection onto $\{\diag(B)\}$. The following lemma tells us that to see such a map is dominant, it suffices to check dominance in the fiber of $(A,B)$'s with a fixed $\diag(B)$.

\begin{lem}[Dominance is open] \label{z}
Let $\X, \Y, \Z$ be irreducible varities and suppose we have a morphism 
\begin{center}\begin{tikzcd}
\Z \times \X \arrow{rr}{\varphi} \arrow{dr}[swap]{p} & & \Z \times \Y \arrow{dl}{q} \\
& \Z
\end{tikzcd}
\end{center}
commuting with projection onto $\Z$. We think of this as a family of morphisms $\X \to \Y$ parameterized by points $z \in \Z$.
If $p^{-1}(z_0) \to q^{-1}(z_0)$ is dominant for some $z_0 \in \Z$, then  $\varphi$ is dominant.
\end{lem}
\begin{proof}
If $p^{-1}(z_0) \to q^{-1}(z_0)$ is dominant, then there exists a point in $q^{-1}(z_0) = z_0 \times \Y$ over which the fiber of $\varphi$ has dimension $\dim z_0 \times \X - \dim z_0 \times \Y = \dim \Z \times \X - \dim \Z\times \Y$. By upper semicontinuity of fiber dimension, we conclude that the general fiber of $\varphi$ has dimension at most $\dim \Z \times \X - \dim \Z \times \Y$. But by dimension counting, equality must hold and the dimension of the image must equal the dimension of $\Z \times \Y$, so $\varphi$ is dominant.
\end{proof}

We will apply Lemma \ref{z} when $\Z = \{\diag(B)\} = \bigoplus_{i=1}^k H^0(\pp^1, \O_{\pp^1}(b_{i,k+1 - i}))$. We let $p: \mathsf{LU} \to \Z$ be the projection that sends $(A, B)$ to $\diag(B)$. Define $\SUT = p^{-1}(0, \ldots, 0) \subset \mathsf{LU}$. In other words, $\SUT$ is the subspace of $(A, B)$ where $A$ is lower triangular and $B$ is $\mathsf{S}$trictly $\mathsf{U}$pper triangular. The space $\SUT$ will play the role of $\mathcal{X}$ in the following application of Lemma \ref{z}. The space $\{P_1(t), \ldots, P_{k}(t)\} = \bigoplus_{i=0}^{k-1}H^0(\pp^1, \O_{\pp^1}(\delta + (k-i)m))$ will play the role of $\Y$. Writing $\{P_1(t), \ldots, P_{k}(t)\} \subset H^0(\pp^1, \Sym^k V \otimes \O_{\pp^1}(\delta))$ for the subspace where $P_{0}(t) = 0$, we have that $\phi^k$ sends $\SUT$ to $\{P_1(t), \ldots, P_{k}(t)\}$.

\begin{lem} \label{simp}
If the restriction of $\phi^k$ to $\SUT \to  \{P_1(t), \ldots, P_{k}(t)\}$ is dominant, then $\phi^k$ is dominant.
\end{lem}
\begin{proof}
Factor $\phi^k$ as in \eqref{factor}. The second map is dominant, so it suffices to show that the first map is dominant. Let $\mathcal{Z} = \{\diag(B)\}$ and let $\mathcal{X} = \SUT$ so that $\mathsf{LU} = \Z \times \X$. Now we apply Lemma \ref{z} with $z_0 = (0, \ldots, 0) \in \Z$. The map $p^{-1}(z_0) = \SUT \to q^{-1}(z_0)$ is the restriction of $\phi^k$ to  $\SUT \to  \{P_1(t), \ldots, P_{k}(t)\}$. In particular, if this map is dominant, then so is $\phi^k$.
\end{proof}

\subsection{The differential of $\phi^k$ restricted to $\SUT$}
Recall that one way to show that a map between smooth, irreducible varieties $\varphi: \X \to \Y$ is dominant is to show that its differential is surjective at some point $p \in \X$. This is because the kernel of the differential at $p$ is the tangent space to the fiber through $p$. Thus, if the differential is surjective at $p$, we learn that 
\[\dim \varphi^{-1}(\varphi(p)) \leq \dim \ker \mathrm{d} \varphi = \dim \X - \dim \Y.\]
By semicontinuity of fiber dimension, the general fiber has dimension at most $\dim \X - \dim \Y$. This implies that the image of $\varphi$ has dimension at least $\dim \Y$, so equality holds and $\varphi$ is dominant.

From now on, we always assume $A$ is lower triangular and $B$ is strictly upper triangular. By $T_{(A,B)}$ we shall mean the tangent space $T_{(A,B)} \SUT$ (which is of course isomorphic to the affine space $\SUT$). 
Let
\[[x^i]: H^0(\pp^1, \Sym^k V \otimes \O_{\pp^1}(\delta)) \to H^0(\pp^1, \O_{\pp^1}(\delta + (k - i)m))\] 
be the projection map defined by
\[[x^i](P(x, y)) = P_i(t).\]
Given $p = (A, B) \in \SUT$, we write
\[\mathrm{d}\phi_p^k: T_p \to \{P_1(t), \ldots, P_k(t)\} = \ker [x^0] \subset H^0(\pp^1, \Sym^k V \otimes \O_{\pp^1}(\delta))\]
for the differential of $\phi^k$ restricted to $\SUT$ at $p$. 
Notice that we are identifying $T_{\phi^k(p)} \ker [x^0]$ with the vector space 
$\ker [x^0]$.

Our goal is to show that $\mathrm{d}\phi_p^k$ surjects onto $\ker[x^0] = \{P_1(t), \ldots, P_k(t)\}$ for general $p \in \SUT$. Once this is established, the discussion and the beginning of this section together with Lemma \ref{simp} will prove Lemma \ref{mainl}.

\begin{lem} \label{sq}
For general $p \in \SUT$, the composition
\begin{equation} \label{dc} ([x^1]  \oplus [x^k]) \circ \d\phi^k_p: T_{p} \to \{P_1(t), P_k(t)\}
\end{equation}
is surjective.
\end{lem}
\begin{proof}
The map in \eqref{dc} is the differential of the map $(A, B) \mapsto (P_1(t), P_k(t))$.
Because $(A, B) \in \SUT$, the polynomials $P_1$ and $P_k$ are given by a simple formula, namely
\begin{equation} \label{p1} P_1(t) = B_{1,k-1}B_{2,k-2} \cdots B_{k-1,1} \cdot A_{k,k}
\end{equation}
and
\begin{equation} \label{pk} P_k(t) = A_{1,k} A_{2,k-1} \cdots A_{k,1}.
\end{equation}
In words, the first term of the output is the product of entries on the super anti-diagonal with the lower right entry; the last term is the product of the entries on the anti-diagonal. All of the entries involved have nonnegative degrees so they can be chosen freely. The result then follows from Lemma \ref{mult} below about the differential of multiplication maps.
\end{proof}

\begin{lem} \label{mult}
Let $d_1, d_2 \geq 0$ and let $\mu: H^0(\O_{\pp^1}(d_1)) \oplus H^0(\O_{\pp^1}(d_2)) \to H^0(\O_{\pp^1}(d_1 + d_2))$ be the multiplication map $(Q_1, Q_2) \mapsto Q_1Q_2$.  The differential $\mathrm{d}\mu$ is surjective at general $(Q_1, Q_2)$. More generally, the differential of an $n$-fold product map $(Q_1, \ldots, Q_n) \mapsto Q_1 \cdots Q_n$ is surjective at a general point.
\end{lem}
\begin{proof}
This is immediate in characteristic zero: since $\mu$ is surjective, the differential is surjective at a general point. For an argument in any characteristic, observe that the differential at $(Q_1, Q_2) = (t^{d_1}, s^{d_2})$ is given by $(Q_1', Q_2') \mapsto t^{d_1} Q_2' + s^{d_2}Q_1'$, which is clearly surjective. The rank of the differential is lower semicontinuous, so it has full rank at a general point in the domain. The claim for an $n$-fold product follows by induction and the fact that the multiplication map is surjective.
\end{proof}

For the next step, we first recall a basic fact from linear algebra that can be used to show that a linear map $T: V \to W$ is surjective. Let $W' \subset W$ be any subspace. If the composition $V \xrightarrow{T} W \to W/W'$ is surjective and $W' \subset T(V)$, then $T$ is surjective. 
Therefore, in light of Lemma \ref{sq} (letting $V = T_p$, $W = \{P_1(t),\dots,P_k(t)\}$, and $W'$ be the subspace below), to show that $\mathrm{d}\phi^k$ surjects onto $\ker[x^0]$,
it suffices to show that the subspace
\[\ker ([x^1] \oplus [x^k]) = \{P_1(t), \ldots, P_k(t) : P_1 = P_k = 0\} \subset \ker[x^0]\] 
is contained in the image of $\mathrm{d}\phi^k$.

We are going to prove this by induction, but we prove a slightly stronger statement that will play well with our inductive step.

Given $p = (A, B) \in \SUT$
let $T_{p}' \subset T_{p}$ denote the subspace of the tangent space where the entries on the anti-diagonal and super anti-diagonal vanish. Precisely, let $A_{i,j}'$ and $B_{i,j}'$ be the coordinates on the tangent space, so we represent a tangent vector as $(A + \epsilon A', B + \epsilon B')$ where $\epsilon^2 = 0$.
Then $T_{p}'$ is the subspace defined by 
$A_{i,k+1-i}' = 0$ and $B_{i,k-i}' = 0$.
Consider the restriction of the differential to this subspace
\begin{equation} \label{tp} 
T_{p}' \to \{P_1(t), \ldots, P_k(t)\}.
\end{equation}
Explicitly, the differential is given by sending $(A', B') \in T_p'$ to the coefficient of $\epsilon$ in
\[\det((A + \epsilon A')x + (B + \epsilon B')y) \mod \epsilon^2.\]
Considering \eqref{p1}, since the $B_{i,k-i}$ terms are unchanging, we see that the image of this map is certainly contained in the subspace where $\prod_{i=1}^{k-1} B_{k-i,i}$ divides $P_1(t)$.
Similarly, considering \eqref{pk}, since the $A_{i, k-i+1}$ are unchanging, the image is contained in the subspace where $P_k(t)$ vanishes. We will prove that $T'_{p}$ in fact surjects onto the subspace
where $\prod_{i=1}^{k-1} B_{k-i,i}$ divides $P_1(t)$ and $P_k(t)$ vanishes.

\begin{lem} \label{main}
For general $p = (A, B) \in \SUT$, the image of
\begin{equation} \label{me} ([x^1] \oplus \cdots \oplus [x^k]) \circ \mathrm{d}\phi^k_p: T'_p \to \{P_1(t), \ldots, P_k(t)\}
\end{equation}
is the subspace where
$\prod_{i=1}^{k-1} B_{k-i,i}$ divides $P_1(t)$
and
$P_k(t) = 0$.
\end{lem}

It is easy to obtain Lemma \ref{mainl} from Lemma \ref{main}: from Lemma \ref{main}, it of course follows that the image of $\mathrm{d}\phi^k$ contains the further subspace where $P_k(t) = P_1(t) = 0$. Next, we prove Lemma \ref{main} by induction on $k$, but first we must develop a lemma to help with the inductive step. The reader who wishes to see how it will be used first can read as far as the statement of Lemma \ref{is} and then skip to Section \ref{pfm}.

\subsection{Technical lemma for the inductive step}
Let $\k$ be our ground field. 
Given a homogeneous degree $\ell$ polynomial $F \in \k[s,t]$ with distinct roots $p_1, \ldots, p_\ell \in \mathbb{A}^1_{\k} \subset \pp^1$ in the chart $s = 1$, we
define 
\[\Phi_F: H^0(\O_{\pp^1}(d)) \to \k^{\oplus \deg F}\]
via $f \mapsto (f(1,p_1), \ldots, f(1,p_\ell))$. 

Let $p = (A, B) \in S$. Recall that we defined $T'_p \subset T_p$ to be the subspace of the tangent space to $\SUT$ where $A_{i,k+1-i}' = B_{i,k-i}' = 0$. We define $T_p'' \subset T_p'$ to be the further subspace where $A_{k,k}' = 0$. Finally, we define the subspace $T^{\llcorner}_p \subset T''_p$ to be the further subspace where $A_{k,1}' = 0$ and $A_{i,j}' = B_{i,j}' = 0$ for all $i,j$ such that $i \leq k - 1$ and $j \geq 2$. The notation derives from the fact this corresponds to the entries wrapping around the lower left corner of the matrix below. The subspace $T^{\llcorner}_p$ corresponds to perturbing only entries in the first column and last row (besides $A_{k,1}$ and $A_{k,k}$ and $B_{k-1,1}$). These entries have green {\color{green!70!black} $+\epsilon$}'s in the matrix below to indicate that they deform when moving within $T^{\llcorner}_p$.

\begin{lem} \label{is}
Suppose $k \geq 3$. For $2 \leq i \leq k-1$, let $F_i$ be a collection of polynomials of degree $b_{k-i, i}$ all with distinct roots also distinct from each other. Set $F = \prod_{i=2}^{k-1} F_i$.
Let $G$ be a polynomial of degree $a_{k,1}$ with distinct roots, distinct from all the $F_i$. Let $\SUT' \subset \SUT$ be the subset of $(A, B)$ with $B_{k-i,i} = F_i$ and $A_{k,1} = G$. For general $p \in \SUT'$, the map
\begin{equation} \label{themap} \left(\bigg (\Phi_{F} \circ [x^2]\bigg) \oplus
\bigg(\bigoplus_{i=2}^{k-1} \Phi_G \circ [x^i] \bigg)
\right) \circ \mathrm{d}\phi_p^k : T^{\llcorner}_{p} \to \k^{\oplus \deg F} \oplus  \bigoplus_{i=2}^{k-1} \k^{\oplus \deg G}    
\end{equation}
is surjective.
\end{lem}

\begin{proof}
 Let $r$ be the largest number such that $b_{k-r,1} \geq 0$.
Surjectivity of the map \eqref{themap} is an open condition, so it suffices to find some $p = (A, B) \in \SUT'$ where the map is surjective.
Our choice is a point where $Ax + By$ has the following form. In order to declutter notation, each $A$ entry below is understood to be multiplied by $x$ and each $B$ entry to be multiplied by $y$.
\begin{equation*}
\left(\begin{matrix}
0 & 0 & 0 & 0 & \cdots & 0 & 0 & {\color{red} B_{1,k-1}} & A_{1,k} \\
0 & 0 & 0 & 0 & \cdots & 0 & {\color{red} B_{2,k-2}} & A_{2,k-1} & A_{2,k} \\
0 & 0 & 0 & 0 & \cdots & {\color{red} B_{3,k-3}} &  A_{3,k-2} & 0 & A_{3,k} \\
 \vdots & &  & &  & & &  & \vdots  \\
 0 & 0 & 0 &  & \cdots &  & 0 & 0 & A_{k-r-1,k} \\
 B_{k-r,1} {\color{green!70!black} + \epsilon} & 0 & 0 &   & &  & 0 & 0 & 0\\
  \vdots & \vdots &  \vdots & &  & & \vdots & \vdots & \vdots  \\ 
 B_{k-3,1} {\color{green!70!black} + \epsilon } & 0 & {\color{red} B_{k-3,3}} & A_{k-3,4}  & \cdots & 0 & 0 & 0  & 0 \\
B_{k-2,1} {\color{green!70!black} + \epsilon } & {\color{red} B_{k-2,2}} & A_{k-2,3} & 0 & \cdots & 0 & 0 & 0  & 0 \\
B_{k-1,1} {\color{green!70!black}} & 0 & 0 & 0  & \cdots & 0 & 0 & 0 & A_{k-1,k} \\
{\color{orange} A_{k,1}} & A_{k,2}{\color{green!70!black} + \epsilon }& A_{k,3}  {\color{green!70!black} + \epsilon } & A_{k,4} {\color{green!70!black} + \epsilon } & \cdots & A_{k,k-3} {\color{green!70!black} + \epsilon} &
A_{k,k-2} {\color{green!70!black} + \epsilon } & A_{k, k-1} {\color{green!70!black} + \epsilon } & A_{k,k}
\end{matrix} \right).
\end{equation*}
The red entries (the $B_{k-i,i} = F_i$) and the orange entry ($A_{k,1} = G$) are fixed. The black entries are general. The green {\color{green!70!black} $+\epsilon$} indicates that the tangent vector at that entry is permitted to be nonzero (for example, $A_{k,i} + \epsilon$ is shorthand for $A_{k,i} + \epsilon A_{k,i}'$ where $A_{k,i}'$ corresponds to the tangent vector changing $A_{k,i}$ in that direction.) The term {\color{green!70!black} $+\epsilon$} marks the subspace we call $T^{\llcorner}_p$.

To compute the differential of $\phi^k$ at $p$, we must find the $\epsilon$ coefficient of 
\begin{equation} \label{ed} \det((A + \epsilon A')x + (B + \epsilon B')y) \mod \epsilon^2
\end{equation}
for arbitrary $(A', B') \in T^{\llcorner}_p$.
To do so, we consider the Laplace expansion of the determinant \eqref{ed} across the bottom row. We now show that the first and last term in this expansion have vanishing $\epsilon$ coefficient. To show this for the first term, observe that $A_{k,1}' = 0$ and the minor obtained by deleting the last row and first column does not contain an $\epsilon$ term. Similarly, to show this for the last term, observe that $A_{k,k}' = 0$ and the determinant of the minor obtained by deleting the last row and last column has no $\epsilon$ term; in that minor, $B_{k-1,1}$ is the only nonzero entry in its row, so it must be used and so no other $\epsilon$ terms appear. Consequently, the $\epsilon$ coefficient in \eqref{ed} is a sum over $2 \leq i \leq k-1$ of
the $\epsilon$ coefficients the following quantity: $\pm (A_{k,i} + \epsilon A_{k,i}')$ times
$x$ times the minor obtained by deleting the last row and $i$th column.

Recall that the
$\epsilon$ coefficient of \eqref{ed} is a polynomial in $x, y$ whose coefficients are polynomials in the entries of $A, B, A',$ and $B'$. Since $\epsilon^2 = 0$, this expression is a linear form in the  entries of $A'$ and $B'$. It will be useful to us to consider the $\epsilon$ coefficient of \eqref{ed} as a linear form in the entries of $A',B'$ whose coefficients are polynomials in $x,y$. The coefficients of these polynomials in $x,y$ are polynomials in the entries of $A,B$. The argument has three main steps:
\begin{enumerate}
    \item First, we claim that the coefficient of $A_{k,i}'$ in the $\epsilon$ coefficient of \eqref{ed}
    is a polynomial of degree $k - i+1$ in $x$. Since the degree of monomials of $x$ that appear decreases as $i$ increases, this will show that \eqref{themap} has an ``upper triangular" form.
    \item Next, we give an explicit formula for the contribution of $A_{k,i}'$ to the coefficient of each $x^{k-i+1}$ and $x^2$.
    \item From (1) and (2), we reduce to showing surjectivity of the restriction of $\Phi_G \circ [x^{k-i+1}] \oplus \Phi_{F_i} \circ [x^2]$ to the subspace $\{A_{k,i}',B_{k-i,1}'\} \subset T^{\llcorner}_p$ where all but these two entries vanish.
    \end{enumerate}

One convenient fact about our chosen point $p$ is that the second to last row has only two nonzero entries, $B_{k-1,1}$ and $A_{k-1,k}$, so precisely one of them appears in each term of the determinant.

\medskip
\textit{Claim:} Any term involving $A_{k-1,k}$ and $A_{k,i}'$ only contributes to the $x^2$ coefficient and the contribution is zero unless $i \leq r$. 

Proof: Seeing that such a term must contribute to $x^2$ is easy: once we cross out the bottom two rows, $i$th column, and rightmost column, all the $A$ terms lie strictly below the diagonal. Thus, we obtain a term divisible by $y^{k-2}$. Moreover, if $i > r$, then the resulting matrix obtained from crossing out these rows and columns is block lower triangular, but the lower left $(i-1) \times (i-1)$ block has a row of zeros, so the determinant vanishes. (Recall that here and throughout, upper and lower triangular are with meant with respect to the anti-diagonal.)

\medskip
\textit{Step 1: Show the coefficient of $A_{k,i}'$ is a polynomial of degree  at most $k - i+1$ in $x$.}

Proof: 
Note that we always have $i \leq k - 1$ so $k - i +1 \geq 2$. Looking at the second to last row, we have two choices: we can use $B_{k-1,1}$ or we can use $A_{k-1,k}$. If we use $A_{k-1,k}$, the claim says the only output is in the $x^2$ coefficient.
If we use $B_{k-1,1}$, we must take the determinant of the matrix obtained by crossing out the bottom two rows, leftmost column and $i$th column. This matrix is block lower triangular, and the lower left block involving the columns to the left of the $i$th column
is lower triangular with red $B$'s along the diagonal. Hence, the red $B$'s in the columns to the left of the $i$th column must occur in the determinant. Together with $B_{k-1,1}$, this forces
$i-1$ copies of $y$, so the $x$ degree is at most $k - i + 1$.

\medskip
\textit{Step 2a: Identify the coefficient of $x^{k-i+1}$ in the coefficient of $A_{k,i}'$ when $i < k - 1$.}
We first show that for $2 \leq i < k - 1$, the $x^{k - i+1}$-coefficient contributed by $A_{k,i}'$ must come from using $B_{k-1,1}$ in the second to last row. Indeed, the claim above says that if we use $A_{k-1,k}$, then it only contributes to $x^2$.
Now cross out the bottom two rows, left column and $i$th column. As mentioned above, we must use the red entries on the diagonal between columns $1$ and $i$. To get the maximal $x$ degree, imagine setting the rest of the $B$'s in the matrix equal to zero. When we do this, the matrix becomes lower triangular and the remaining entries on the diagonal are the $A_{k+1-j,j}$ with $j > i$. Combined with Step 1,
it follows that $[x^{k-i+1}] \circ \d\phi^k(A',B')$ has the form
\[\pm {\color{green!70!black} A_{k,i}'} \cdot  B_{k-1,1}{\color{red} B_{k-2,2} \cdots B_{k-(i-1),i-1} }  \cdot A_{k-i,i+1} \cdots A_{1,k}
+ (\text{terms with $A_{k,j}'$ for $j < i$}). \]

\medskip
\textit{Step 2b: Identify the coefficient of $x^2$.}
We also need to track the contribution of every term to the $x^2$ coefficient. Recall that we are expanding the determinant along the bottom row, and the $\epsilon$ coefficient in \eqref{ed}
is a sum over $2 \leq i \leq k-1$ of
the $\epsilon$ coefficients of $\pm (A_{k,i} + \epsilon A_{k,i}')$ times
$x$ times the minor obtained by deleting the last row and $i$th column.
The contribution to the coefficient of $x^2$ from the $i$th term depends on whether $i \leq r$ or not, which we treat in cases (i) and (ii) below. 

\medskip
(i) The case $i \leq r$. 
Notice that all entries in the rightmost column are $A$'s so
it necessarily contributes an $x$.  Suppose we use $A_{j,k}$ with $j < k - 1$ from the last column. To see the contribution, we cross out the $j$th row, last row, last column and $i$th column. Now consider the $(k - j)$th column: the only nonzero entry not already crossed out is $A_{j+1,k-j}$, but using it would increase the $x$ degree to $3$ or more.
Thus, we must use $A_{k-1,k}$ from the last column. To calculate this contribution, cross out the bottom two rows, last column and $i$th column. We have accounted for two $x$'s already, so we can set the remaining $A$'s equal to zero. The result is a block upper triangular matrix where the upper right block has red $B$'s along the diagonal and the lower left $(i - 1 \times i - 1)$ block has red $B$'s on the subdiagonal. To get nonzero determinant in this lower left block, we must use $B_{k - i, 1} + \epsilon B_{k-i,1}'$ times all the remaining red entries. The contribution is therefore the $\epsilon$ coefficient of
\[\pm (A_{k,i} + {\color{green!70!black} \epsilon A_{k,i}'})(B_{k - i,1} + {\color{green!70!black} \epsilon B_{k-i,1}'}) \cdot A_{k-1,k} \cdot \prod_{\substack{2 \leq j \leq k - 1 \\ j \neq i}} {\color{red} B_{k-j,j}},\]
which is
\begin{equation} \label{x21}
\pm  ({\color{green!70!black} A_{k,i}'} B_{k-i,1} + {\color{green!70!black} B_{k-i,1}'}A_{k,i}) \cdot A_{k-1,k} \cdot \prod_{\substack{2 \leq j \leq k - 1 \\ j \neq i}} {\color{red} B_{k-j,j}}. 
\end{equation} 

\medskip
(ii) The case $i > r$. We must use one of $A_{k-1,k}$ or $B_{k-1,1}$ from the second to last row. By the claim, since $i > r$, the contribution from using $A_{k-1,k}$ vanishes. Thus, 
we must use $B_{k-1,1}$ from the second to last row. To find this contribution, we cross out the first column, $i$th column, and last two rows. The result is a block lower triangular matrix. The lower left $(i - 2) \times (i - 2)$ block is lower triangular with red $B$'s on the diagonal. The upper right $(k - i) \times (k - i)$ block has red $B$'s on the superdiagonal and $A_{k - i, k}$ in its lower right entry. The contribution is therefore
\begin{equation} \label{x22} \pm  {\color{green!70!black} A_{k,i}'} \cdot B_{k-1,1} \cdot A_{k-i,k} \prod_{\substack{2 \leq j \leq k - 1 \\ j \neq i}} {\color{red} B_{k-j,j}}.
\end{equation}

\medskip
In summary, we have $[x^2] \circ \d\phi^k(A',B')$ is the sum of \eqref{x21} over $2 \leq i \leq r$ plus the sum of  \eqref{x22} over $r + 1 \leq i \leq k-1$. 

\medskip
\textit{Step 3: Study evaluation maps and reduce to subspaces for $\{A_{k,i}',B_{k-i,1}'\}$.}
Our calculations above allow us to explicitly determine the map
\begin{equation} \label{phif} \Phi_F \circ [x^2] \circ \d\phi^k_p : T_p^{\llcorner} \to \k^{\oplus \deg F} = \bigoplus_{i=2}^{k-1} \k^{\oplus \deg F_i},
\end{equation}
where we have used the fact that the $F_i$ have no common roots to decompose the target into the direct sum of spaces corresponding to each $F_i$. Notice that in each of \eqref{x21} and \eqref{x22}, all but one of the red $B$'s appears. This means that when we apply $\Phi_F$, the contributions break up nicely. In particular,
\eqref{phif} sends a tangent vector $(A_{k,2}', \ldots, A_{k,k-1}', B_{k-2,1}' \ldots, B_{k-r,1}') \in T_p^{\llcorner}$ to the column vector where the $i$th component is $\Phi_{F_i}$ applied to \eqref{x21} if $2 \leq i \leq r$ or to \eqref{x22} if $r+1 \leq i \leq k-1$. Moreover, the $F_i$'s and black entries have no common zeros since the black entries are general. Therefore, up to rescaling basis vectors of the target (according to the values of the other terms at the roots of the $F_i$), the map \eqref{phif} is given by
\begin{gather*}
\Phi_{F_2}( {\color{green!70!black} A_{k,2}'} B_{k-2,1} + {\color{green!70!black} B'_{k-2,1}}A_{k,2}) \\
 \vdots \\
\Phi_{F_r}({\color{green!70!black} A_{k,r}'} B_{k-r,1} + {\color{green!70!black} B'_{k-r,1}}A_{r,2}) \\
 \Phi_{F_{r+1}} ({\color{green!70!black} A_{k,r+1}'}) \\
\vdots \\
 \Phi_{F_{k-1}}({\color{green!70!black} A_{k,k-1}'}).
\end{gather*}

Meanwhile,  up to rescaling basis vectors of the target,  $\Phi_G \circ [x^{k-i+1}]$ has the form
\[(A_{k,2}', \ldots, A_{k,k-1}', B_{k-2,1}' \ldots, B_{k-r,1}') \mapsto \Phi_G(A_{k,i}') + \Phi_G(\text{terms with $A_{k,j}'$ for $j < i$}).\]
It follows that $\bigoplus_{i=2}^{k-1}\Phi_G \circ [x^{k-i+1}]$ is block triangular and the blocks on the diagonal are given by 
\[A_{k,i}' \mapsto \Phi_G(A_{k,i}').\]
Thus, to show that \eqref{themap} is surjective, it suffices to see that for each $i$, the map
\[({\color{green!70!black} A_{k,i}'}, {\color{green!70!black}B_{k-i,1}'}) \mapsto 
\begin{cases} ( \Phi_{F_i}({\color{green!70!black}A'_{k,i} }B_{k-i,1} + {\color{green!70!black}B_{k-i,1}'} A_{k,i}), \Phi_G({\color{green!70!black}A_{k,i}'}))  &\text{if $2 \leq i \leq r$} \\
( \Phi_{F_i}({\color{green!70!black}A_{k,i}'}), \Phi_G({\color{green!70!black}A_{k,i}'}))
&\text{if $r+1 \leq i \leq k-1$}
\end{cases}
\]
is surjective. If $r + 1 \leq i \leq k - 1$, then by the definition of $r$, we have $b_{k-i,1} < 0$. Hence,
we have 
\[\deg A_{k,i}' = a_{k,i} > a_{k,i} + b_{k-i,1} = a_{k,1} + b_{k-i,i} = \deg G + \deg F_i,\]
and surjectivity follows. On the other hand, if $2 \leq i \leq r$, then we have
\[\deg A_{k,i}' + \deg B_{k-i,1}' = a_{k,i} + b_{k-i,1} = a_{k,1} + b_{k-i,i} = \deg G + \deg F_i \]
Hence, the case $2 \leq i \leq r$ follows from the lemma below.

Recall that $\deg G = a_{k,1} \leq a_{k,i} = \deg A'_{k,i}$.
\end{proof}

\begin{lem}
Fix homogeneous polynomials $F, G \in \k[s,t]$ with no common roots. Choose $a,b \geq 0$ such that $a + b = \deg(F) + \deg(G)$ and $a \geq \deg(G)$. For general homogeneous polynomials $A,B \in \k[s,t]$ of degrees $a,b$, the map
\[H^0(\pp^1, \O(a)) \oplus H^0(\pp^1, \O(b)) \to \k^{\oplus \deg F} \oplus \k^{\oplus \deg G}\]
given by 
\[(A', B') \to (\Phi_{F}(A'B + B'A), \Phi_G(A')) \]
is surjective. 
\end{lem}
\begin{proof}
The map in question is the differential of the map $(A, B) \mapsto (\Phi_{F}(AB), \Phi_G(A)$). Since surjectivity of the differential is an open condition, it suffices to show that this map has surjective differential at some point $(A, B)$. To prove this, we show that the composition 
\begin{equation} \label{compo}
\begin{tikzcd}
H^0(\pp^1, \O(\deg G)) \oplus H^0(\pp^1, \O(\deg A - \deg G)) \oplus H^0(\pp^1, \O(\deg B)) \arrow{d}  \\
H^0(\pp^1, \O(\deg A)) \oplus H^0(\pp^1, \O(\deg B)) \arrow{d} \\
\k^{\oplus \deg F} \oplus \k^{\oplus \deg G}
\end{tikzcd}
\end{equation}
given by
\begin{equation}
(C, D, B) \mapsto (CD, B) \mapsto (\Phi_{F}(CDB), \Phi_G(CD))
\end{equation}
has surjective differential at some point. We will compute the differential of \eqref{compo} at a point $(C, D, B) = (G, D, B)$ where no pair of $G, D$ and $B$ have a common root. Let us write elements of the tangent space as $(C', D', B')$.
We claim that the image of the subspace $\{(0, D', B')\}$ contains the subspace $\k^{\oplus \deg F} \oplus 0 \subset 
\k^{\oplus \deg F} \oplus \k^{\oplus \deg G} $.
 Indeed, the derivative of \eqref{compo} sends
\[(0, D', B') \mapsto (\Phi_{F}(GDB' + GBD'), 0).\]
Since $G$ and $F_i$ have no common roots, this differs from $(\Phi_{F}(B'D + BD'), 0)$ by a rescaling of basis vectors on the target. Since $B$ and $D$ have no common roots, the map $(D', B') \mapsto B'D + D' B$ surjects onto the space of polynomials of degree $\deg B + \deg D  = \deg B + \deg A - \deg G = \deg F$. The evaluation map $\Phi_{F}: H^0(\pp^1, \O(\deg F)) \to \k^{\oplus \deg F}$ is surjective. This proves the claim.

Finally, we show that differential sends the subspace $\{(C', 0, 0)\}$ surjectively onto the quotient of the tangent space to the target given by $(\k^{\oplus \deg F} \oplus \k^{\deg G})/(\k^{\oplus \deg F} \oplus 0)$.
Indeed, the map to this quotient is given by $(C', 0, 0) \mapsto \Phi_G(C'D)$. This is just a rescaling of basis vectors of $(C', 0, 0) \mapsto \Phi_G(C')$, which is surjective because $C'$ is an arbitrary polynomial of the same degree as $G$. 
\end{proof}

\subsection{Proof of Lemma \ref{main}} \label{pfm}
We proceed by induction on $k$. 
If $k = 2$, then $T_p''$ is the subspace corresponding to the entry $A_{22}'$. Using the formula for the determinant in \eqref{sf}, we see that the differential sends $A_{22}' \mapsto P_2(t) = 0$ and $P_1(t) = B_{11} A_{22}'$, as desired.
We now assume $k \geq 3$ and that we have proved the lemma for all smaller values of $k$.

First we observe that the composition of \eqref{me} with
projection onto $P_1(t)$ hits every multiple of $\prod_{i=1}^{k-1} B_{k-i,i}$.
Indeed, consider a tangent vector $(A', B')$ where all entries of $A'$ and $B'$ vanish except for $A_{k,k}'$. Then $\d\phi_p^k(A', B') =  (\prod_{i=1}^{k-1} B_{k-i,i} \cdot A_{k,k}')xy^{k-1}$. Varying $A_{k,k}'$, we see that $P_1(t)$ can be any multiple of $\prod_{i=1}^{k-1} B_{k-i,i}$.

Thus, we are reduced to showing that the subspace $\{0, P_2(t), \ldots, P_{k-1}(t), 0\}$ lies in the image of \eqref{me}. The subspace $T''_p \subset T'_p$ where $A_{k,k}' = 0$ is always sent to this subspace.
It thus suffices to show that the map
\[\d_p := ([x^2] \oplus \cdots \oplus [x^{k-1}]) \circ \d \phi^k_p: T''_p \rightarrow \{P_2(t), \ldots, P_{k-1}(t)\}\]
is surjective.
The key insight is to
decompose $T''_p = T^{k-1}_p \oplus T^{\llcorner}_p$ where
\begin{enumerate} 
\item $T^{k-1}_p$ is the ``inductive subspace" defined by the vanishing of entries in the first column and last row.
Let 
\[s: \{\text{$k \times k$ matrices} \} \to \{\text{$k-1 \times k-1$ matrices} \}
\]
be the map that sends a matrix to the one obtained by deleting the last row and first column.
Given a pair of matrices $p = (A, B)$, we define $s(p) = (s(A), s(B))$. There is a natural identification of $T_{p}^{k-1}$ with $T'_{s(p)}$. 
\item $T^{\llcorner}_p$ is the ``expanding subspace'' corresponding to the entries in the first column and last row. Note that by the definition of $T''_p$ we always have $A_{k,1}' = A_{k,k}' = B_{k-1,1}' = 0$. The space $T^{\llcorner}_p$ is the subspace corresponding to the green {\color{green!70!black} $+\epsilon$} in the matrix in the proof Lemma \ref{is}.
\end{enumerate}
Our goal is then to show that $\d_p(T^{k-1}_p)$ and $\d_p(T^{\llcorner}_p)$ span $\{P_2(t), \ldots, P_{k-1}(t)\}$ for a general point $p \in \SUT$.

Now let $p = (A, B) \in \SUT$ be a general point. In particular, we can take the entries $B_{k-i,i}$ for $2 \leq i \leq k-1$ and $A_{k,1}$ to all have distinct roots and be pairwise coprime.
Let $F = \prod_{i=2}^{k-1} B_{k-i,i}$ and $G = A_{k,1}$.
By Lemma \ref{is}, the map
\begin{equation} \label{ev1} \left(\bigg(\Phi_{F} \circ [x^2] \bigg)\oplus
\bigg(\bigoplus_{i=2}^{k-1} \Phi_G \circ [x^i] \bigg) 
\right) \circ \mathrm{d}_p: T^{\llcorner}_p \to \k^{\oplus \deg F} \oplus \bigoplus_{i=2}^{k-1} \k^{\oplus \deg G} 
\end{equation}
is surjective. Notice that the kernel of
\begin{equation} \label{myevals} \bigg(\Phi_{F} \circ [x^2]\bigg)\oplus
\bigg(\bigoplus_{i=2}^{k-1} \Phi_G \circ [x^i] \bigg): \{P_2(t), \ldots, P_{k-1}(t)\} \to \k^{\oplus \deg F} \oplus \bigoplus_{i=2}^{k-1} \k^{\oplus \deg G}  
\end{equation}
is the subspace such that $G \vert P_2(t), \ldots, P_{k-1}(t)$ and $F \vert P_2(t)$. This subspace has advantageous properties for our inductive argument.

\begin{lem} \label{ql}
Suppose $q \in \SUT$ is a general point subject to the constraint that the entries besides $A_{k,1}$ in the bottom row vanish. Then $\d_q(T_q^{k-1}) \subset \{P_2(t), \ldots, P_{k-1}(t)\}$ is the subspace such that $G \vert P_2(t), \ldots, P_{k-1}(t)$ and $F \vert P_2(t)$. 
\end{lem}
\begin{proof} 
Because of the vanishing entries in the bottom row,
for any $(A', B') \in T_q^{k-1}$, we have 
\begin{align*} \d\phi^k_q(A', B') &= \text{$\epsilon$ coefficient of} \det((A + \epsilon A')x + (B + \epsilon B')y) \mod \epsilon^2 \\
&= A_{k,1}x \cdot \text{$\epsilon$ coefficient of} \det(s(A + \epsilon A')x + s(B + \epsilon B')y) \mod \epsilon^2 \\
&= Gx \cdot \d\phi^{k-1}_{s(q)}(s(A', B')).
\end{align*}
Since $q$ is general subject to the vanishing constraints in its bottom row,
$s(q)$ is general. The map $(A', B') \mapsto s(A', B')$ defines an isomorphism $T_q^{k-1} \to T_{s(q)}'$.
By induction on $k$, we know that $\d\phi^{k-1}_{s(q)}( T_{s(q)}')$ is the subspace of polynomials of degree $k-1$
where the constant coefficient vanishes, $\prod_{i=2}^{k-1} B_{k-i,i} = F$ divides the coefficient of $x$, and the coefficient of $x^{k-1}$ vanishes. Multiplying such polynomials by $Gx$, we obtain arbitrary polynomials of degree $k$ such that the constant and $x$ coefficients vanish, $F$ divides the coefficient of $x^2$, $G$ divides all coefficients, and the coefficient of $x^k$ vanishes.
\end{proof}

If we knew that Lemma \ref{is} held at $q$ with this vanishing along the bottom row, then
Lemma \ref{ev1} would show that 
 $\d_q(T^{k-1}_q)$ is the kernel of \eqref{ev1}, and we would conclude that 
$\d_q(T^{k-1}_q)$
and
 $\d_q(T^{\llcorner}_q)$ span $\{P_2(t), \ldots, P_{k-1}(t)\}$. However, the $q$ with this vanishing along the bottom row are not general, so we cannot invoke Lemma \ref{is} at $q$.
Instead, we are going to consider a family of points $p(h)$ parameterized by $h \in \mathbb{A}^1_{\k} \smallsetminus 0$ and follow $\d_{p(h)}(T^{k-1}_{p(h)})$ and $\d_{p(h)}(T^{\llcorner}_{p(h)})$ as $h \to 0$. We'll show that the flat limits of these two subspaces span. It then follows that the two subspaces span for general $h$.

First we prove a general fact about the flat limit of a family of linear maps, and then give a toy example showing how our specialization argument will go, using this lemma.

\begin{lem} \label{dh}
Suppose $D_h: V \to W$ is a family of linear maps between two vector spaces which is represented by matrices whose entries are polynomials in $h$. Then 
\[D_0(V) \subseteq \lim_{h \to 0}D_h(V) .\] 
\end{lem}
\begin{proof}
Let $\alpha: \mathbb{A}^1 \times V \to \mathbb{A}^1 \times V \times W$ be given by
$(h, v) \mapsto (h, v, D_h(v))$. 
 Because $\alpha$ is continuous we have
\[
\alpha(\mathbb{A}^1 \times V) =
\alpha\left(\overline{(\mathbb{A}^1 \smallsetminus 0) \times V}\right) \subseteq \overline{\alpha((\mathbb{A}^1 \smallsetminus 0) \times V)}.  \]
By definition, the fiber over $0 \in \mathbb{A}^1_{\k}$ on the right-hand side is the flat limit $\lim_{h \to 0}D_h(V)$. The fiber over $0 \in \mathbb{A}^1_{\k}$ on the left is $D_0(V)$.
\end{proof}

\begin{rem}
For a toy example of how our specialization argument will go, suppose we were trying to show that a linear map $D_h: V \to W$ between two-dimensional vector spaces is surjective for general $h$, and we are given the following set of facts. Let $v_1,v_2$ and $w_1,w_2$ be bases of the source and target respectively.
\begin{itemize}
    \item $D_0(\langle v_1 \rangle)$ equals the kernel of $W \to W/\langle w_1 \rangle$ (this is the analogue of Lemma \ref{ql});
    \item for $h \neq 0$, we have $D_h(\langle v_2 \rangle) = D_1(\langle v_2 \rangle)$ independent of $h$ (this is the analogue of \eqref{analogue} in Lemma \ref{flatlimit}).
    \item when $h = 1$, the composition $D_1(\langle v_2 \rangle) \subset W \to W/\langle w_1 \rangle$ is surjective (this is the analogue of Lemma \ref{is});
\end{itemize}
For a concrete example of a family of maps satisfying these conditions, take
\[
    D_h = \left(\begin{matrix} 1 & 0 \\ h & h \end{matrix}\right).
\]

The first fact, combined with Lemma \ref{dh}, shows that $\langle w_1 \rangle \subseteq \lim_{h \to 0}D_h(\langle v_1 \rangle)$. The second fact shows that $\lim_{h \to 0} D_h(\langle v_2 \rangle) = D_1(\langle v_2 \rangle)$. The third fact implies that these two subspaces $\lim_{h \to 0}D_h(\langle v_1 \rangle)$ and $\lim_{h \to 0} D_h(\langle v_2 \rangle)$ span $W$. Because spanning is an open condition, this implies that $D_h$ is surjective for a general $h$.

Notice that we never had to understand $D_h(\langle v_1 \rangle)$; we only relied on our knowledge of $D_0(\langle v_1 \rangle)$.  In our scenario, $\langle v_1 \rangle$ will play the role of $T_p^{k-1}$ and $\langle v_2 \rangle$ will play the role of $T^{\llcorner}_p$. By arguing in an analogous fashion below, we'll never need to understand $\d_{p(h)}(T_{p(h)}^{k-1})$ for general $h$ --- it will suffice to understand $\d_{p(0)}(T_{p(0)}^{k-1})$.
\end{rem}

Returning to the argument, let $p \in \SUT$ be a general point and define $p(h)$ to be the family obtained by scaling the entries $A_{k,2}, \ldots, A_{k,k}$ in the bottom row of $p$ by $h$. So $p(1) = p$ and $q = p(0)$ satisfies the hypotheses of Lemma \ref{ql}. Let $\lim_{h \to 0} \d_{p(h)}(T^{k-1}_{p(h)})$ denote the flat limit of the subspaces $\d_{p(h)}(T^{k-1}_{p(h)}) \subset \{P_2(t), \ldots, P_{k-1}(t)\}$.

\begin{lem}
We have
\begin{equation} \label{lim1} 
\d_{q}(T^{k-1}_q)\subseteq
\lim_{h \to 0} \d_{p(h)}(T^{k-1}_{p(h)}).
\end{equation}
In particular, $\lim_{h \to 0} \d_{p(h)}(T^{k-1}_{p(h)})$ contains the kernel of \eqref{ev1}.
\end{lem}
\begin{proof}
As $h$ varies, the spaces $T_{p(h)}^{k-1}$ are all canonically identified.
The family of maps 
\[\d_{p(h)}: T_{p(h)}^{k-1} \to \{P_2(t), \ldots, P_{k-1}(t)\}\] 
can thus be represented by a family of matrices where the entries are polynomials in $h$.
The result now follows from Lemma \ref{dh}.
\end{proof}

Next we determine the flat limit of  the subspaces $\d_{p(h)}(T^{\llcorner}_{p(h)})$.
\begin{lem}
\label{flatlimit}
For $h \neq 0$, we have $\d_{p(h)}(T^{\llcorner}_{p(h)}) = \d_p(T^{\llcorner}_{p})$. Hence,
\begin{equation} \label{lim2} \lim_{h \rightarrow 0}\d_{p(h)}(T^{\llcorner}_{p(h)}) = \d_p(T^{\llcorner}_p)
\end{equation}
\end{lem}
\begin{proof} 
Let us write $p(h) = q + hMx$ where $M$ is the matrix with entries $0, A_{k,2}, \ldots, A_{k,k}$ along the bottom row and zeros everywhere else.
As stated above, we have $p(1) = p$ and $p(0) = q$.
For any $v = (A', B') \in T_{p(h)}^{\llcorner}$, we have that 
\[\d_{p(h)}(v) = \epsilon \text{ coefficient of } \det(p(h) + \epsilon (A'x + B'y)) \mod \epsilon^2.\]
For $h \neq 0$, let us define $q^*(h)$ to be the point obtained from $q$ by replacing $A_{k,1}$ with $h^{-1}A_{k,1}$. Then for $h \neq 0$ we have that $p(h)$
is obtained from $q^*(h) + Mx$ by scaling the bottom row by $h$.
In particular,
\[\d_{p(h)}(v) = h \cdot (\epsilon \text{ coefficient of } \det(q^*(h) + Mx + \epsilon h^{-1}(A'x + B'y)) \mod \epsilon^2)\]
for any $v = (A', B') \in T^{\llcorner}_{p(h)}$.
In turn, we claim that $\epsilon$ coefficient above equals the $\epsilon$ coefficient of $\det(q + Mx + \epsilon h^{-1}(A'x + B'y)) = \det (p + \epsilon h^{-1}(A'x + B'y))$. To see the claim, notice that
all $\epsilon$'s occur in either the bottom row or left column. In particular, when we expand the determinant along the bottom row, $A_{k,1}$ never appears multiplied by $\epsilon$.
In summary, we have shown that for $h \neq 0$,
\begin{equation} \label{dv} \mathrm{d}_{p(h)}(v) = h\cdot \d_p(h^{-1}v).
\end{equation}
As $h$ varies, the subspaces $T^{\llcorner}_{p(h)}$ are all canonically identified.
In particular, \eqref{dv} shows that for $h \neq 0$, we have
\begin{equation} \label{analogue} \d_{p(h)}(T_{p(h)}^{\llcorner}) = \d_p(T^{\llcorner}_p).
\end{equation}
That is, the family of subspaces on the left is constant.
Hence, the flat limit as $h \to 0$ is also the space $\d_p(T_p^{\llcorner})$.
\end{proof}

We are now ready to conclude the proof of Lemma \ref{main}.
Recall that at the start of this subsection, we reduced to 
showing that the subspace $\{0, P_2(t), \ldots, P_{k-1}(t), 0\}$ lies in the image of \eqref{me}.
Lemma \ref{is} tells us that
\eqref{ev1} is surjective. 
Furthermore, by Lemma \ref{ql}, the kernel of
\[\{P_2(t), \ldots, P_{k-1}(t)\} \to
\k^{\oplus \deg F} \oplus \bigoplus_{i=2}^{k-1}\k^{\oplus \deg G}
\]
is equal to $\d_q(T^{k-1}_q)$. In other words, we have shown that there is a surjection
\[\d_p(T^{\llcorner}_{p}) \to \{P_2(t), \ldots, P_{k-1}(t)\}/\d_q(T^{k-1}_{q}).\]
In particular,
$\d_p(T^{\llcorner}_{p})$ and $\d_q(T^{k-1}_{q})$ span all of $\{P_2(t), \ldots, P_{k-1}(t)\}$.
On its own, this is not sufficient, since $p$ and $q$ are different points. However, utilizing \eqref{lim1} and \eqref{lim2}, this turns into the fact that
\[\lim_{h \rightarrow 0}\d_{p(h)}(T^{\llcorner}_{p(h)}) \qquad \text{and} \qquad \lim_{h \rightarrow 0} \d_{p(h)}(T^{k-1}_{p(h)})\]
span all of $\{P_2(t), \ldots, P_{k-1}(t)\}$. Since spanning is an open property, the spaces $\d_{p(h)}(T^{\llcorner}_{p(h)})$ and $\d_{p(h)}(T^{k-1}_{p(h)})$ must span for general $h$.

\section{Proofs of Corollary \ref{thecor} and Theorem \ref{planecurves}} \label{corp}

\subsection{Proof of Corollary \ref{thecor}}
If $U^{\vec{e}}(C)$ is non-empty, then $U^{\vec{e},\vec{f}}(C)$ is non-empty for some $\vec{f}$. By Theorem \ref{maint} condition (1), we have $f_i - e_i \geq 0$ and by condition (2), we have $f_i - e_i \geq e_{i+1} - e_i - m$. Hence, $f_i - e_i \geq \max\{0, e_{i+1} - e_i - m\}$. Combining this with condition (3), we find that
\[\delta = \sum_{i=1}^k f_i - e_i \geq \sum_{i=1}^{k-1} \max\{0, e_{i+1} - e_i - m\}. \]

Conversely, if the above condition holds, we must show there exists some $\vec{f}$ so that conditions (1) -- (3) in Theorem \ref{maint} are satisfied. Indeed, we can take $f_i = e_i + \max\{0, e_{i+1} - e_i - m\}$ for $i \leq k-1$ and $f_k = e_k + \delta - \sum_{i=1}^{k-1} \max\{0, e_{i+1} - e_i - m\}$.

\subsection{Proof of Theorem \ref{planecurves}}
Theorem \ref{planecurves} follows quickly from specializing Theorem \ref{maint} to the case when $\delta = 0$ and $m = 1$. Indeed, when $\delta = 0$, conditions (1) and (3) imply that $\vec{e} = \vec{f}$. Moreover, condition (2) then implies that $e_{i+1} - e_i \leq 1$ for all $i$. Since $\vec{e} = \vec{f}$, the formula for the dimension of $U^{\vec{e}, \vec{f}}(C)$ simplifies as follows:
\begin{align*}
\dim U^{\vec{e}}(C) &= \dim U^{\vec{e},\vec{e}}(C) = g - u(\vec{e}) - u(\vec{e}) + h^1(\pp^1, \End(\vec{e})) + h^1(\pp^1, \End(\vec{e})(1)) \\
&= g - h^1(\pp^1, \End(\O(\vec{e})) +  h^1(\pp^1, \End(\vec{e})(1)).
\end{align*}
To conclude, notice that 
\[h^1(\pp^1, \O(e_i - e_j + 1)) - h^1(\pp^1, \O(e_i - e_j)) = \begin{cases} -1 & \text{if $e_i - e_j \leq -2$} \\ 0 & \text{otherwise.} \end{cases} \]
Finally, we note that the splitting type of a line bundle is independent of the choice of projection point. If $C$ is general, projection from a general point is a general curve in $\mathbb{F}_1$ not meeting the directrix.

\section{Further discussion} \label{discuss}

\subsection{Constraints on splitting types of line bundles}
\label{firsthalfsection5}
As discussed at the beginning of this article, curves with distinguished line bundles may have unusual Brill--Noether theory. The first example is covers $\alpha \colon C \rightarrow \PP^1$; the distinguished line bundle is $\alpha^* \cO_{\PP^1}(1)$. The next example is curves on Hirzebruch surfaces, which we have shown have even more special Brill--Noether theory; in addition to $\alpha^*\O_{\pp^1}(1)$, we also have the distinguished line bundle $\cO_C(\Delta)$. These observations naturally lead to the following question: given a cover of $\PP^1$ and a line bundle with known splitting type, what can we say about the splitting types of the other line bundles on the curve?

In general this is an extremely difficult question; we first simplify the situation by assuming that the cover is {\em primitive}, meaning it does not factor through a nontrivial proper subcover. For example, if $k$ is prime, or if the monodromy group of $\alpha$ is the full symmetric group $S_k$, then $\alpha$ is primitive. As shown in \cite{scrinvts} and \cite{succmin}, the splitting types of structure sheaves of primitive curves seem to be significantly simpler than those of imprimitive curves. Therefore for the rest of this section, we focus on primitive curves.

In this section, we prove a general bound (Theorem~\ref{generalbound}) on the splitting types of triples of line bundles $L_1, L_2, L_3$ such that $L_1 \otimes L_2 \simeq L_3$. This provides non-trivial information about how the presence of a special line bundle, say $L_2$ with known splitting type, constrains the possible splitting types of $L_1$ and $L_3$. For example, if $L_2 = \O$, then $L_1 \cong L_3$ and this bound leads to non-trivial constraints on the splitting type of $L_1$. This bound is not special to curves on Hirzebruch surfaces, and an application of this bound gives a new perspective on Conditions (1) and (2) in the case of primitive curves on Hirzebruch surfaces.

Pick any cover $\alpha \colon C \rightarrow \PP^1$ and any line bundles $L_1, L_2, L_3 \in \Pic(C)$ such that $L_1 \otimes L_2 \simeq L_3$. Let $\vec{d}$, $\vec{e}$, and $\vec{f}$ be the splitting types of $\alpha_*L_1$, $\alpha_*L_2$, and $\alpha_*L_3$ respectively. The degrees of these line bundles imposes one important constraint on the splitting types:
\begin{equation}
\label{degreebound}
    \sum_{i = 1}^k (d_i + e_i - f_i) = -(g + k - 1).
\end{equation}
Additionally, we have the following strong bound, which is essentially \cite[Theorem 1.4]{scrinvts}, with minor modifications.

\begin{thm}
\label{generalbound}
If $\alpha$ is primitive, then $f_{i + j - k} \geq d_{i} + e_{j}$ for all $i,j$ with $1 \leq i + j - k \leq k$.
\end{thm}

\begin{example}
There exist imprimitive covers for which the inequalities in Theorem~\ref{generalbound} do not hold. For example, let $k = 4$. Let $\alpha_1 \colon C_1 \rightarrow \PP^1$ and $\alpha_2 \colon C_2 \rightarrow \PP^1$ be two hyperelliptic curves whose ramification divisors intersect trivially. We obtain a degree $4$ cover $\alpha \colon C \rightarrow \PP^1$ via the pullback:
\[
\begin{tikzcd}
C \arrow{r}{} \arrow{dr}{\alpha} \arrow[swap]{d}{} & C_1 \arrow{d}{\alpha_1} \\
C_2 \arrow{r}[swap]{\alpha_2}& \PP^1
\end{tikzcd}
\]
Let $(L_1,L_2,L_3) \coloneqq (\cO_C,\cO_C,\cO_C)$ and write
\[
(\alpha_*\cO_C)^{\vee} = \cO_{\PP^1} \oplus \cO_{\PP^1}(a_1) \oplus \O_{\pp^1}(a_2) \oplus \cO_{\PP^1}(a_3).
\]
By Riemann-Roch (as in Remark \eqref{nec}), we see that $\det (\alpha_i)_* \O_{C_i} = -g_i - 1$, where $g_i$ is the genus of $C_i$. Because this pushforward admits an $\O_{\pp^1}$ summand, it follows that
$(\alpha_i)_* \cO_{C_i} = \cO_{\PP^1} \oplus \cO_{\PP^1}(-g_i - 1)$. Because the ramification divisors intersect trivially,
\[
    \alpha_* \cO_C = (\alpha_1)_* \cO_{C_1} \otimes (\alpha_2)_* \cO_{C_2}.
\]
Without loss of generality suppose $g_1 \leq g_2$; then $a_1 = g_1+1$ and $a_2 = g_2+1$. The conclusion of Theorem~\ref{generalbound}, applied in the case $i = j = 3$, says that
\[
    a_2 \leq 2a_1
\]
but choosing $g_2$ to be sufficiently large relative to $g_1$ ensures that this inequality does not hold.
\end{example}

\begin{proof}[Proof of Theorem~\ref{generalbound}]
Choose an isomorphism $L_1 \otimes L_2 \rightarrow L_3$ and push forward to obtain $\alpha_*L_1 \otimes \alpha_*L_2 \rightarrow \alpha_*L_3$. Choose a point $\infty \in \PP^1$, and choose a coordinate $t$ on $\mathbb{A}^1 = \PP^1 \setminus \{ \infty \}$, i.e., an isomorphism $\PP^1 \setminus \{ \infty \}  \cong \Spec \CC[t]$.  Then the splitting $\alpha_* L_1 = \cO_{\PP^1}(e_1) \oplus \cdots \oplus \cO_{\PP^1}(e_{k})$ induces a splitting of the $\CC[t]$-module $\Gamma(\mathbb{A}^1, \alpha_*L_1)$ into $\Gamma(\mathbb{A}^1, \alpha_*L_1) = \CC[t]x_1 \oplus \cdots \oplus \CC[t] x_{k}$, where here $x_i$ is a generator of the $i$th summand  of $\Gamma(\mathbb{A}^1, \alpha_*L_1)$. Similarly, we write $\Gamma(\mathbb{A}^1, \alpha_*L_2) = \CC[t]y_1 \oplus \cdots \oplus \CC[t] y_{k}$ and $\Gamma(\mathbb{A}^1, \alpha_*L_3) = \CC[t]z_1 \oplus \cdots \oplus \CC[t] z_{k}$. Now $x_1$, \dots, $x_{k}$ are a basis for $K(C)$ as a $K(\PP^1)$-module (where $K(\cdot)$ indicates the function field), and similarly for the $y_j$ and $z_{\ell}$.

Thus, we obtain a map
\[
    (\CC[t]x_1 \oplus \cdots \oplus \CC[t] x_{k}) \otimes (\CC[t]y_1 \oplus \cdots \oplus \CC[t] y_{k}) \rightarrow (\CC[t]z_1 \oplus \cdots \oplus \CC[t] z_{k})
\]
given exactly by multiplication in $K(C)$. The map from the product of the $i$th summand and the $j$th summand to any summand $\cO(f_{\ell})$ where $f_{\ell} < d_i + e_j$ must vanish, because $\Hom(\cO(d_i) \otimes \cO(e_j), \cO(f_{\ell})) = 0$ in that case.  

The key ingredient in the proof of Theorem~\ref{generalbound} is the observation that ``too many'' elements in the multiplication table of a field cannot all be zero; to quantify this, we make use of a theorem of 
Bachoc, Serra and Z\'emor \cite{bsz} in additive combinatorics, which we now describe. Suppose $L/K$ is a finite field extension of fields, and $A$ and $B$ are two nonzero $K$-subvector spaces  of $L$. Let  $AB$ be the $K$-vector space spanned by elements of the form $ab$ for $a\in A$ and $b \in B$.  Given $\dim_K A$ and $\dim_K B$, Bachoc, Serra, and Z\'emor give a bound on  $\dim_K AB$. Define $\Stab(AB)$ to be the set of elements in $L$ stabilizing $AB$ when acting by multiplication in the field:
$$
\Stab(AB) := \{ x \in L : x (AB) \subset AB \}.$$
Notice that $\Stab(AB)$ is an intermediate field extension of $L/K$: it contains $K$, and is clearly closed under field operations. Theorem 3 of \cite{bsz} states that 

\[
\dim_K AB \geq \dim_K A + \dim_K B - \dim_K \Stab(AB).
\]

Now suppose for contradiction that there are $i$ and $j$ such that $f_{i + j - k} < d_{i} + e_{j}$.  Let $K = K(\PP^1)$, $L = K(C)$, $A$ be the sub-$K$-vector space of $L$ generated by  $x_k$, \dots, $x_{i}$, and $B$ the vector subspace of $L$ generated by $y_k$, \dots, $y_{j}$. Then $AB$ is contained in the vector subspace of $L$ generated by $z_k$, \dots, $z_{i + j - k + 1}$, so 
$$k - (i + j - k + 1) + 1 \geq \dim_K AB \geq (k - i + 1) + (k - j + 1) - \dim_K \Stab{AB},$$
from which $\Stab(AB) \geq 2$.  Because $\Stab(AB)$ is an intermediate field of $L/K$, strictly containing $K$ and strictly contained in $L$, we obtain a contradiction.
\end{proof}

Conditions (1) and (2) of Theorem~\ref{maint} for primitive covers can now be obtained via an application of Theorem~\ref{generalbound}, as we explain. 
For a primitive cover on a Hirzebruch surface, let
\[
    (L_1,L_2,L_3) \coloneqq (\cO(\Delta), L, L(\Delta))
\]
Let $\vec{e}$ be the splitting type of $L$ and let $\vec{f}$ be the splitting type of $L(\Delta)$. We claim that the splitting type of $\cO(\Delta)$ is: 
\[
   \vec{d} = (-(k-1)m - \delta,\dots,-2m-\delta,-m,0).
\]
This follows from pushing forward the sequence of sheaves on $\pp V$
\[0 \rightarrow \O_{\pp V}(-C + D) \rightarrow \O_{\pp V}(D) \rightarrow \iota_*\O_C(\Delta) \rightarrow 0\]
along $\gamma$ and using the fact that
\[\gamma_* \O_{\pp V}(D) = \gamma_* \O_{\pp V}(H - \delta F) = \O_{\pp^1}(-\delta) \otimes V\]
and
\begin{align*}
R^1\gamma_*\O_{\pp V}(-C + D) &= (\gamma_* \omega_\gamma \otimes \O_{\pp V}(C - D))^\vee = (\gamma_* \O_{\pp V}(-H - D + kH + \delta F - D) )^\vee \\
&= (\gamma_* \O_{\pp V}((k - 3)H + (\delta + 2m) F) = (\O_{\pp^1}(\delta + 2m) \otimes \Sym^{k - 3} V)^\vee.
\end{align*}
Applying Theorem~\ref{generalbound} to $i = k$ in the given triple $(L_1,L_2,L_3)$, we obtain
$f_{j} \geq d_k + e_j = e_j$,
recovering Condition (1). Similarly, applying Theorem~\ref{generalbound} to $i = k - 1$ gives the inequality $f_{j-1} \geq d_{k-1} + e_{j} = - m + e_j$, recovering Condition (2). Note that Condition (3) always holds for degree reasons.

\begin{question}
\label{MQuestion}
For which curves $C$ does there exist a line bundle $M$ such that degree constraints together with Theorem~\ref{generalbound} applied to the triple $(M,L, M\otimes L)$ cuts out the set of splitting types of line bundles occurring on $C$?
\end{question}

Precisely, we say that $M$ cuts out the splitting types of line bundles occurring on $C$ if the possible splitting types of $L$ and $M \otimes L$ satisfying the constraints of Theorem~\ref{generalbound} and the degree constraint given in 
\eqref{degreebound}, projected to the splitting types of $L$, are precisely the splitting types of line bundles occurring on $C$. For general primitive curves on Hirzebruch surfaces, the discussion above shows that $M = \cO_C(\Delta)$ answers Question~\ref{MQuestion} affirmatively.

\subsection{Conjectures on splitting types of curves and line bundles}

Given a degree $k$ cover $\alpha \colon C \rightarrow \PP^1$, let $\vec{a} = (a_1, \ldots, a_{k-1})$ be such that
\[
(\alpha_*\cO_C)^{\vee} = \cO_{\PP^1} \oplus \cO_{\PP^1}(a_1) \oplus \dots \oplus \cO_{\PP^1}(a_{k-1}).
\]
These are typically called \emph{scrollar invariants} or the \emph{splitting type of the cover}.
Accounting for duals and the shift in indexing, Theorem~\ref{generalbound}, applied to the case $(\cO_C,\cO_C,\cO_C)$, implies that $a_{i + j} \leq a_i + a_j$ for primitive covers. The second author, in \cite{scrinvts}, conjectures that these are the \emph{only} constraints. Since this comes from applying Theorem~\ref{generalbound} to the triple  $(\cO,\cO,\cO)$, we call the following conjecture the $\cO$-$\cO$ Conjecture.

\begin{conj}[$\cO$-$\cO$ Conjecture \cite{scrinvts}]
There exists a primitive cover $C \rightarrow \PP^1$ of splitting type $\vec{a}$ if and only if $a_{i + j} \leq a_i + a_j$ for all $i,j$.
\end{conj}

In low degree, the parametrizations of degree $k$ covers imply the truth of this conjecture in these cases. For all larger $k$, a ``positive proportion'' of these splitting types is exhibited in \cite{scrinvts}, giving evidence towards the $\cO$-$\cO$ Conjecture. We now extend this conjecture to splitting types of line bundles. We ask: which tuples $(\vec{a},\vec{e})$ arise from a pair $(\alpha \colon C \rightarrow \PP^1, L)$ where $\vec{a}$ is the splitting type of a primitive cover $\alpha$ and $\vec{e}$ is the splitting type of $L$? 

One necessary constraint on $(\vec{a},\vec{e})$ is the following. If $(\vec{a},\vec{e})$ is realized by a primitive cover, then Theorem~\ref{generalbound} applied the triple $(\cO_C,L,L)$ implies that $e_{i + j} \leq a_i + e_j$ for all $i,j$. In the absence of any other information, one might wonder if this necessary constraint, along with the constraints of the $\cO$-$\cO$ conjecture, is also sufficient.

\begin{conj}[$\cO$-$L$ Conjecture]
\label{loconjecture}
There exists a primitive cover $C \rightarrow \PP^1$ of splitting type $\vec{a}$ and a line bundle $L$ with splitting type $\vec{e}$ if and only if
\[
    e_{i + j} \leq a_i + e_j
\]
and
\[
    a_{i + j} \leq a_i + a_j
\]
for all $1 \leq i, j \leq i + j \leq k$.
\end{conj}

We have stated this as a conjecture on little evidence in the hopes that it may also motivate the search for a counterexample.

\medskip
We say that a curve is \emph{abundant} if it has line bundles exhibiting every splitting type specified by the $\cO$-$L$ Conjecture; note that a curve is abundant precisely if the line bundle $\cO$ answers Question~\ref{MQuestion} affirmatively. The $\cO$-$L$ conjecture can be shown to be true by proving abundant curves exist in a number of cases.  Firstly, somewhat trivially, all genus $0$ and $1$ covers are abundant. Secondly, observe that every hyperelliptic curve is abundant. The splitting type of a hyperelliptic cover of genus $g$ is $(a_1) = (g+1)$. Clifford's theorem implies that the curve has a line bundle of splitting type $(e_1,e_2)$ if and only if $e_2 - e_1 - 1 \leq g$; in other words, if $e_2 \leq e_1 + a_1$. Moreover, as proven by the first author, every trigonal curve is abundant\cite[Theorem 1.1]{trig}. This is in contrast to the situation for general $k$-covers for $k > 3$ and large $g$, as the following proposition shows. 

\begin{prop}
For $k > 3$ and $g \geq 2(k-1)$, a general $k$-cover is not abundant. 
\end{prop}
\begin{proof}
Write 
\[
    g = (k-1)(\delta-1) + \ell
\]
for integers $1 \leq \delta$ and $0 \leq \ell \leq k - 2$. Because $g \geq 2(k-1)$, we must have that $\delta \geq 3$. The splitting type of a general $k$-cover of genus $g$ is balanced and we may write it as
\[
    (\delta,\dots,\delta,\delta+1,\dots,\delta+1)
\]
where the last $\ell$ coordinates are equal to $\delta + 1$. Consider the splitting type $\vec{e} = (0,\dots,0,\delta,\delta)$; clearly, it satisfies the conditions of the $\cO$-$L$ conjecture. By \cite{refined}, the splitting type $\vec{e}$ occurs on a general degree $k$ cover if and only if
\[
    u(\vec{e}) = \sum_{i < j}\max\{0,e_j - e_i - 1\} \leq g.
\]
We would like to show that $u(\vec{e}) > g$, i.e. that
\[
(\delta - 1)2(k - 2) >  (k-1)(\delta-1) + \ell.
\]
Subtracting $(k-1)(\delta-1)$ from both sides and applying the bound $\ell \leq k - 2$, it suffices to show that
\[
    (\delta - 1)(k - 3) > k-2
\]
The truth of this inequality follows from $\delta \geq 3$ and $k > 3$. 
\end{proof}

Despite the fact that a general $k$-cover is not abundant when $g$ is large, if $(k-1) \mid g$ it is possible to produce abundant curves whose splitting types match that of a general $k$-cover.

\begin{prop}
A general curve on $\PP^1 \times \PP^1$ is abundant, hence the $\cO$-$L$ Conjecture is true when $\vec{a} = (\delta,\dots,\delta)$ for $\delta > 0$.
\end{prop}
\begin{proof}
It suffices to assume that $k \geq 3$. First note that the set of $\vec{e}$ such that
\[
    e_{i+j} \leq a_i + e_j
\]
for all $i,j$ is precisely the set of $\vec{e}$ such that $e_{k} \leq e_1 + \delta$. To prove the proposition, we'll show that a generic curve on a Hirzebruch surface with splitting type $\vec{a}$ is abundant. In this case, $m = 0$. Observe that the linear inequality of Corollary~\ref{thecor} is equivalent to the condition that
\[
    e_k \leq e_1 + \delta.
\]
Thus, the set of splitting types achieved by a general curve on a Hirzebruch surface with splitting type $\vec{a}$ is precisely the set of splitting types conjectured by the $\cO$-$L$ conjecture.
\end{proof}

Our proof exhibited that a general curve on $\PP^1 \times \PP^1$ of fixed class is abundant. One may ask: is \emph{every} curve on $\PP^1 \times \PP^1$ abundant? If $\delta = 1$ then $g = 0$, so the curve is abundant. As previously discussed, if $k \leq 3$, then the curve is abundant as well.

\begin{prop}
A general curve on a Hirzebruch surface not meeting the directrix is abundant, hence the $\cO$-$L$ Conjecture is true when $\vec{a} = (m,\dots,(k-1)m)$ for $m > 0$.
\end{prop}
\begin{proof}
It suffices to assume that $k \geq 3$. First note that the set of $\vec{e}$ such that
\[
    e_{i+j} \leq a_i + e_j
\]
is precisely the set of $\vec{e}$ such that $e_{i+1} \leq e_i + m$ for all $i$. Then we'll show that a generic curve on a Hirzebruch surface with splitting type $\vec{a}$ is abundant. Indeed, in this case, $\delta = 0$, and the linear inequality of Corollary~\ref{thecor} is equivalent to the condition that
\[
    e_{i+1} \leq e_i + m
\]
for all $i$.
\end{proof}

Note that this proposition also follows from the fact that $\cO(\Delta) = \cO$ in this case and the discussion at the end of Section~\ref{firsthalfsection5}. Every curve with splitting type $(m,\dots,(k-1)m)$ lies on a Hirzebruch surface, so a general curve of that splitting type is abundant. It is reasonable to ask: is \emph{every} curve of splitting type $ (m,\dots,(k-1)m)$ abundant? Conversely, one may ask if, for $k > 3$, curves on Hirzebruch surfaces whose splitting types are not $(m,\dots,(k-1)m)$ or $(\delta,\dots,\delta)$ are abundant: the answer is always no, as we show below.

\begin{prop}
For $k > 3$, every curve on a Hirzebruch surface with $m\delta \neq 0$ is not abundant.
\end{prop}
\begin{proof}
In this case, we have $a_i = im + \delta$.
It suffices to exhibit an example of $\vec{e}$ that satisfies the linear inequalities of Conjecture~\ref{loconjecture} but does not satisfy the linear inequalities of Corollary~\ref{thecor}. Take:
\[
    \vec{e} = (0,m+\delta,m+\delta,2m+\delta+1,\dots,2m+\delta+1)
\]
which satisfies the conditions of Conjecture~\ref{loconjecture} if 
\[
    2m+\delta+1 \leq 3m + \delta
\]
and
\[
    2m + \delta + 1 \leq (m + \delta) + (m + \delta)
\]
i.e $1 \leq m, \delta$.
\end{proof}

We conclude this article with one last question.

\begin{question}
For which $\vec{a}$ does there exist an abundant curve with splitting type $\vec{a}$? The discussion above gives two examples of such $\vec{a}$, namely $\vec{a} = (\delta,\dots,\delta)$ for $\delta > 0$ and $\vec{a} = (m,\dots,(k-1)m)$ for $m > 0$. Are there other examples?    
\end{question}

\bibliographystyle{amsplain}
\bibliography{refs}

\end{document}